\documentclass[12pt,a4paper,twoside,reqno]{amsart}

\usepackage[all,cmtip]{xy}
\usepackage{amsmath,amscd,mathscinet}
\usepackage[T1]{fontenc}
\usepackage{amssymb,amsfonts}
\usepackage[driver=pdftex,margin=3cm,heightrounded=true,centering]{geometry}
\usepackage[colorlinks=true,linkcolor=blue,citecolor=blue]{hyperref}
\usepackage{verbatim}

\usepackage{graphicx}

\usepackage{amsthm,slashed}

\theoremstyle{plain}

\newtheorem{thm}{Theorem}[section]
\newtheorem{prop}[thm]{Proposition}
\newtheorem{lemma}[thm]{Lemma}
\newtheorem{cor}[thm]{Corollary}

\theoremstyle{definition}
\newtheorem{dfn}[thm]{Definition}

\theoremstyle{remark} 

\newtheorem{remark}[thm]{Remark}

\theoremstyle{plain}

\usepackage{amscd}

\numberwithin{equation}{section}

\newcommand{\alpheqn}[1][\relax]{
     \refstepcounter{equation}
     \if#1\relax \relax
       \else \label{#1}
     \fi  
     \setcounter{saveeqn}{\value{equation}}%
    \setcounter{equation}{0}%
    \renewcommand{\theequation}{\thealphequation}}
\newcommand{\reseteqn}{\setcounter{equation}{\value{saveeqn}}%
     \renewcommand{\theequation}{\thearabicequation}}

\providecommand{\mathscr}{\mathcal} 
\IfFileExists{mathrsfs.sty}{
\usepackage{mathrsfs}}         
   {
     \IfFileExists{eucal.sty}{
        \usepackage[mathscr]{eucal}} 
   {
   }
}

\newcommand{\sa}{{\operatorname{sa}}}
\newcommand{\Lip}{{\operatorname{Lip}}}

\newcommand{\dist}{{\operatorname{dist}}}

\newcommand{\cd}{\cdot}
\newcommand{\clc}{\cdot\ldots\cdot}
\newcommand{\ot}{\otimes}

\newcommand{\op}{\oplus}

\newcommand{\ci}{\circ}

\newcommand{\ti}{\times}
\newcommand{\nn}{\mathbb{N}}
\newcommand{\zz}{\mathbb{Z}}
\newcommand{\rr}{\mathbb{R}}
\newcommand{\cc}{\mathbb{C}}

\newcommand{\al}{\alpha}
\newcommand{\be}{\beta}
\newcommand{\ga}{\gamma}
\newcommand{\Ga}{\Gamma}
\newcommand{\de}{\delta}
\newcommand{\De}{\Delta}
\newcommand{\ep}{\varepsilon}

\newcommand{\la}{\lambda}

\newcommand{\si}{\sigma}

\newcommand{\te}{\theta}

\newcommand{\ze}{\zeta}
\newcommand{\pa}{\partial}

\newcommand{\ov}{\overline}
\newcommand{\C}[1]{\mathcal{#1}}
\newcommand{\G}[1]{\mathfrak{#1}}
\newcommand{\T}[1]{\textup{#1}}

\newcommand{\E}[1]{\emph{#1}}
\newcommand{\B}[1]{\mathbb{#1}}

\newcommand{\fork}[2]{\left\{ \begin{array}{#1} #2 \end{array} \right.} 

\newcommand{\ma}[2]{\left(\begin{array}{#1} #2 \end{array} \right)}

\newcommand{\su}{\subseteq}

\newcommand{\inn}[1]{\langle #1 \rangle}

\newcommand{\sem}{\setminus}

\usepackage{a4wide}
\renewcommand{\leq}{\leqslant}
\renewcommand{\geq}{\geqslant}

\usepackage{thmtools}
\declaretheorem[style=theorem,name={Theorem}]{theoremletter}

\title{The Podle\'s spheres converge to the sphere}
\author{Konrad Aguilar, Jens Kaad, and David Kyed}

\address{Department of Mathematics, Pomona College, 610 N. College Ave., Claremont, CA 91711}
\email{konrad.aguilar@pomona.edu}

\address{Department of Mathematics and Computer Science,
The University of Southern Denmark,
Campusvej 55, DK-5230 Odense M,
Denmark}

\email{kaad@imada.sdu.dk}

\address{Department of Mathematics and Computer Science,
The University of Southern Denmark,
Campusvej 55, DK-5230 Odense M,
Denmark}

\email{dkyed@imada.sdu.dk}



\begin{document}

\begin{abstract}
We prove that the Podle{\'s} spheres $S_q^2$ converge in quantum Gromov-Hausdorff distance to the classical $2$-sphere as the deformation parameter $q$ tends to $1$. Moreover, we construct a $q$-deformed analogue of the fuzzy spheres, and prove that they converge to $S_q^2$ as their linear dimension tends to infinity, thus providing a quantum counterpart to a classical result of Rieffel.
\end{abstract}

\maketitle

\tableofcontents

\section{Introduction}
The theory of  $C^*$-algebras provides a vast noncommutative generalisation of the theory of locally compact topological spaces,  and by imposing suitable additional structures, one obtains noncommutative (or quantum) analogues of more sophisticated topological spaces. Two very successful examples of this phenomenon are the theory of quantum groups \cite{kustermans-vaes-C*-lc, wor:cpqgrps} which are noncommutative analogues of topological groups, and Connes' notion of  spectral triples \cite{Con:NCG}, which are the noncommutative counterparts to (spin) Riemannian manifolds. In the same vein, it is also very natural to ask for a noncommutative generalisation of ordinary metric spaces, and Rieffel's seminal work \cite{Rie:MSS, Rie:GHD} provides a very satisfactory answer to this question. Rieffel's fundamental insight is that the right noncommutative  counterpart to a metric on a compact topological  space is a certain densely defined seminorm on a unital $C^*$-algebra, and he dubbed these structures \emph{compact quantum metric spaces}. Over the past 20 years, ample examples of compact quantum metric spaces have emerged, and the theory has been developed in several different directions through the works of many hands; see
 \cite{LatAgu:AF, BMR:DSS, HSWZ:STC, Ker:MQG, Lat:QGH,   Li:GH-dist,    Rie:MSA} and references therein. One of the most important features of the theory is that it admits a generalisation of the classical Gromov-Hausdorff distance \cite{ edwards-GH-paper, gromov-groups-of-polynomial-growth-and-expanding-maps}, known as the \emph{quantum Gromov-Hausdorff distance} \cite{Rie:GHD}, which allows one to study the theory of quantum metric spaces from an analytic point of view, and thus ask questions pertaining to continuity and convergence. As an example of this, Rieffel showed in \cite{Rie:GHD} that the  noncommutative tori $A_{\theta}$ admit a natural quantum metric structure and that they vary continuously in the deformation parameter $\theta$ with respect to the quantum Gromov-Hausdorff distance, and in \cite{Rie:MSG} that the so-called fuzzy-spheres also admit a compact quantum metric structure with respect to which they converge to the classical 2-sphere $S^2$ as their linear dimension tends to infinity; for many more examples in this direction see for instance \cite{aguilar:thesis,  kaad-kyed,  Lat:AQQ, LatPack:Solenoids}. \\
The definition of compact quantum metric spaces is inspired by  Connes'  theory of noncommutative geometry, and the latter is therefore a natural source of many interesting  examples, which may be viewed as the noncommutative counterparts of  Riemannian manifolds when these are  considered merely as metric spaces with their  Riemannian metric.  Despite a continuous effort over at least 30 years, it has proven quite difficult to reconcile the theory of quantum groups with Connes' noncommutative geometry, \cite{CoMo:TST, Maj:QNG}. In fact, even for the most fundamental example, Woronowicz' quantum $SU(2)$, there are still several competing candidates for good spectral triples, and it is not known which of these provide quantum $SU(2)$ with a quantum metric space structure, \cite{BiKu:DQQ,DLSSV:DOS,KaSe:TST,KRS:RFH,NeTu:DCQ}. However, just as $SU(2)$ has the classical $2$-sphere as a homogeneous space, its quantised counterpart $SU_q(2)$ also has a quantised $2$-sphere $S_q^2$, known as the standard Podle{\'s} sphere,  as a ``homogenous space'' \cite{Pod:QS}, and the work of D\k{a}browski and Sitarz \cite{DaSi:DSP} provides $S_q^2$ with a spectral triple, which was shown in \cite{AgKa:PSM} to turn $S_q^2$ into a compact quantum metric space. 
This  result provides the first genuine quantum analogue of a Riemannian geometry on $S_q^2$, and 
the most pertinent question to investigate at this point is therefore if the quantised $2$-spheres $S_q^2$  converge to the classical round $2$-sphere as the deformation parameter $q$ tends to 1. The present paper answers this questions in the affirmative:
\begin{theoremletter}[{see Theorem \ref{thm:podles-converging-to-classical}}]\label{mainthm:A}
As $q$ tends to 1, the  Podle{\'s} spheres $S_q^2$  converge to the classical $2$-sphere $S^2$  in the quantum Gromov-Hausdorff distance. 
\end{theoremletter}

One of the main ingredients in the proof of Theorem \ref{mainthm:A} is a quantum analogue of Rieffel's convergence result for fuzzy spheres mentioned above. More precisely,  we introduce a sequence of finite dimensional compact quantum metric spaces $(F^N_q)_{N\in \mathbb{N}}$, which play the role of quantised counterparts to the classical fuzzy spheres, and prove the following result:
\begin{theoremletter}[{see Theorem \ref{thm:fuzzy-to-podles}}]\label{mainthm:B}
For each $q \in (0,1]$ the sequence of quantised fuzzy spheres $\left(F_q^N\right)_{N\in \nn}$ converges to $S_q^2$ with respect to the quantum Gromov-Hausdorff distance.  
\end{theoremletter}
For $q=1$, Theorem \ref{mainthm:B} provides a variation of Rieffel's result \cite[Theorem 3.2]{Rie:MSG}, in that it shows that the finite dimensional quantum metric spaces $F_N^1$ converge to the classical round 2-sphere as $N$ tends to infinity.   Along the way, we also prove (see Proposition \ref{prop:quantum-fuzzy-to-classical-fuzzy}) that  for fixed $N\in \nn$, the 1-parameter family $(F_q^N)_{q\in (0,1]}$ of compact quantum metric spaces vary continuously in the quantum Gromov-Hausdorff distance. \\

The rest of the paper is structured as follows: in Section \ref{sec:prelim} we give a detailed introduction to quantum $SU(2)$, the noncommutative geometry of the standard Podle{\'s} sphere and compact quantum metric spaces, while Section  \ref{sec:quantum-fuzzy-spheres} is devoted  to introducing the quantised fuzzy spheres mentioned above. In Section \ref{sec:qGH-convergence} we  carry out the main analysis and prove our convergence results.

\subsection{{Acknowledgements}}
The authors gratefully acknowledge the financial support from  the Independent Research Fund Denmark through grant no.~9040-00107B  and 7014-00145B. Furthermore, they would like to thank Marc Rieffel for pointing out the reference \cite{Sain:thesis}, and the anonymous referees for their careful reading of the manuscript. 

\section{Preliminaries}\label{sec:prelim}

\subsection{Quantum $SU(2)$}
Let us fix a $q \in (0,1]$. We consider the universal unital $C^*$-algebra $C(SU_q(2))$ with two generators $a$ and $b$ subject to the relations
\[
\begin{split}
& ba = q ab \quad b^* a = q ab^* \quad bb^* = b^* b \\
& a^* a + q^2 bb^* = 1 = aa^* + bb^* .
\end{split}
\]
This unital $C^*$-algebra is referred to as \emph{quantum $SU(2)$} and was introduced by Woronowicz in \cite{Wor:UAC}. Notice here that we are conforming with the notation from \cite{AgKa:PSM,DaSi:DSP} which is also known as Majid's lexicographic convention, see \cite{Maj:NRS}. With these conventions the fundamental corepresentation unitary takes the form
\[
u = \ma{cc}{ a^* & - qb \\ b^* & a } .
\]
We let $\C O(SU_q(2)) \su C(SU_q(2))$ denote the unital $*$-subalgebra generated by $a$ and $b$ and refer to this unital $*$-subalgebra as the \emph{coordinate algebra}. 
The coordinate algebra can be given the structure of a Hopf $*$-algebra where the coproduct $\De \colon \C O(SU_q(2)) \to \C O(SU_q(2)) \ot \C O(SU_q(2))$,  antipode $S \colon \C O(SU_q(2)) \to \C O(SU_q(2))$ and  counit $\epsilon \colon \C O(SU_q(2)) \to \cc$ are defined on the fundamental unitary by
\[
\De(u) = u \otimes u, \quad S(u) = u^* \quad \T{and} \quad \epsilon(u) = \ma{cc}{1 & 0 \\ 0 & 1} .
\]
For $q \neq 1$,  we  also consider the universal unital $*$-algebra $\C U_q( \G{su}(2))$ with generators $e,f,k$ satisfying the relations
\[
k k^{-1} = 1 = k^{-1} k  \, \, , \, \, \, ek = qke \, \, , \, \, \, kf = qfk \, \, \T{ and } \, \, \, \frac{k^2 - k^{-2}}{q - q^{-1}} = fe - ef
\]
and with involution defined by $e^* = f$, $f^* = e$ and $k^* = k$. We refer to this unital $*$-algebra as the \emph{quantum enveloping algebra}. The quantum enveloping algebra also becomes a Hopf $*$-algebra with comultiplication, antipode and counit determined by

\begin{align} \label{eq:envhopf}
 \De(e) &= e \ot k + k^{-1} \ot e &  S(e) &= -q^{-1}e &  \epsilon(e) &= 0 \notag  \\
 \De(f) &= f \ot k + k^{-1} \ot f &   S(f) &= - qf &  \epsilon(f) &= 0   \\
\De(k) &= k \ot k&   S(k) &= k^{-1} &  \epsilon(k) &= 1   \notag
\end{align}
We are here again conforming with the notations from \cite{DaSi:DSP}. The quantum enveloping algebra $\C U_q(\G{su}(2))$ is seen to be isomorphic,  as a Hopf algebra,  to  $\E{\u{U}}_q(\T{sl}_2)$ with generators $E,F,K$ from Klimyk and Schm\"udgen \cite[Chapter 3]{KlSc:QGR}, by using the dictionary $e \mapsto F$, $f \mapsto E$, $k \mapsto K$. For $q = 1$, we furthermore consider the \emph{universal enveloping Lie algebra} $\C U(\G{su}(2))$ with generators $e,f,h$ satisfying the relations
\[
[h,e] = -2e \quad [h,f] = 2f \quad [f,e] = h 
\]
and with involution defined by $e^* = f$, $f^* = e$ and $h^* = h$. It too becomes a Hopf $*$-algebra with comultiplication, antipode and counit given by

\begin{align}\label{eq:envlie}
 \De(e) &= e \ot 1 + 1 \ot e  &  S(e) &= - e  &  \epsilon(e) &= 0 \notag  \\
 \De(f) &= f \ot 1 + 1 \ot f&   S(f) &= - f  &  \epsilon(f) &= 0   \\
\De(h) &= h \ot 1 + 1 \ot h &   S(h) &= - h &  \epsilon(h) &= 0   \notag
\end{align}
Notice that $\C O(SU_q(2))$ agrees with classical coordinate algebra $\C O(SU(2))$ for $q = 1$. However, for the quantum enveloping algebra the relationship is slightly more subtle: we obtain the classical universal enveloping Lie algebra by formally putting $h := 2 \log(k)/\log(q)$ so that $k = e^{\log(q) h/2}$ and then letting $\log(q)$ tend to zero. For more information on these matters, we refer to \cite[Section 3.1.3]{KlSc:QGR}. In order to unify our notation in the rest of the paper we apply the convention that $k = 1 \in \C U(\G{su}(2)) =: \C U_1(\G{su}(2))$. \\

The algebras $ \C O(SU_q(2))$  and $\C U_q(\G{su}(2))$ are linked by  a non-degenerate dual pairing of Hopf $*$-algebras $\inn{\cd,\cd} \colon \C U_q(\G{su}(2)) \ti \C O(SU_q(2)) \to \cc$, which for $q \neq 1$ is given by
\begin{align*}
\inn{k,a} &= q^{1/2}  &    \inn{e,a} &= 0 & \inn{f,a} &=0  \\
\inn{k,a^*} & = q^{-1/2}  &  \inn{e,a^*}&=0  &  \inn{f,a^*} &=0   \\
 \inn{k,b} & = 0 &    \inn{e,b} &= -q^{-1} & \inn{f,b} &=0   \\
\inn{k,b^*} &=0  &  \inn{e,b^*} &=0 & \inn{f,b^*} &= 1  
\end{align*}
In the case where $q = 1$, this pairing is determined by
\begin{align*}
\inn{h,a} &= 1 &    \inn{e,a} &= 0 & \inn{f,a} &=0  \\
\inn{h,a^*} & = -1  &  \inn{e,a^*}&=0  &  \inn{f,a^*} &=0   \\
 \inn{h,b} & = 0 &    \inn{e,b} &= -1 & \inn{f,b} &=0   \\
\inn{h,b^*} &=0  &  \inn{e,b^*} &=0 & \inn{f,b^*} &= 1
\end{align*}
See \cite[Chapter 4, Theorem 21]{KlSc:QGR} for more details. \\

The above dual pairing of Hopf $*$-algebras yields a left and a right action of $\C U_q(\G{su}(2))$ on $\C O(SU_q(2))$. For each $\eta \in \C U_q(\G{su}(2))$ the corresponding (linear) endomorphisms of $\C O(SU_q(2))$ are defined by
\[
\de_\eta(x) := (\inn{\eta,\cd} \ot 1)\De(x) \qquad \T{ and } \qquad \pa_\eta(x) := (1 \ot \inn{\eta,\cd})\De(x), 
\]
respectively. \\

As it turns out, we shall need a rather detailed description of the irreducible  representations of $\C U_q(\G{su}(2))$, and to this end the following notation is convenient:
 for  $q \in (0,1]$ and $n \in \nn$, we define
\[
\inn{n} := \sum_{m = 0}^{n-1} q^{2m} . 
\]
We also put $\inn{0} := 0$. For $q=1$ we of course have $\inn{n}=n$, and for $q \neq 1$, the relationship with the usual $q$-integers $[n] := \tfrac{q^n - q^{-n}}{q - q^{-1}}$ is given by $\inn{n} = q^{n-1} [n]$. \\

For each $n \in \nn_0$ we have an irreducible $*$-representation of $\C U_q(\G{su}(2))$ on the Hilbert space $\cc^{n+1}$ with standard orthonormal basis $\{e_j\}_{j = 0}^n$. This irreducible $*$-representation is given on generators by
\[
\begin{split}
\si_n(k)(e_j) & = q^{j-n/2} \cd e_j \\
 \si_n(e)(e_j) &=  q^{\frac{1 - n}{2}} \sqrt{ \inn{n - j +1} \inn{j} } \cd e_{j-1} \\ 
\si_n(f)(e_j) & = q^{\frac{1-n}{2}} \sqrt{ \inn{n-j} \inn{j+1} } \cd e_{j + 1} 
\end{split}
\]
in the case where $q \neq 1$ and by
\[
\begin{split}
\si_n(h)(e_j) & = (2j -n) \cd e_j \\
\si_n(e)(e_j) &= \sqrt{ (n - j + 1) j } \cd e_{j-1} \\
\si_n(f)(e_j) & = \sqrt{ (n-j) (j+1) } \cd e_{j + 1} 
\end{split}
\]
in the case where $q = 1$, see \cite[Chapter 3, Theorem 13]{KlSc:QGR}. 
The above sequence of irreducible $*$-representations together with the non-degenerate pairing $\inn{\cd,\cd} \colon \C U_q(\G{su}(2)) \ti \C O(SU_q(2))\to \cc$ gives rise to a complete set of irreducible corepresentation unitaries $u^n \in \mathbb{M}_{n+1}( \C O(SU_q(2)))$, $n \in \nn_0$. Indeed, the entries in $u^n$ are characterised by the identity
\[
\si_n(\eta)(e_j) = \sum_{i = 0}^n \inn{ \eta, u^n_{ij}} \cd e_i ,
\]
which holds for all $\eta \in \C U_q(\G{su}(2))$ and all $j \in \{0,1,\ldots,n\}$, see \cite[Chapter 4, Proposition 16 \& 19]{KlSc:QGR} for this. We record that $u^0 = 1$ and $u^1 = u$. Notice here that we are applying a different convention than Klimyk and Schm\"udgen \cite{KlSc:QGR} who denote the unitary corepresentations by $\{t^l\}_{l \in \frac{1}{2} \nn_0}$. The relationship with our notation can be summarised by the identities $u^n_{ij} = t^{n/2}_{i-n/2,j-n/2}$ for $n \in \nn_0$ and $i,j \in \{0,1,\ldots,n\}$. \\

When $q \neq 1$, we obtain from the definition of the irreducible corepresentation unitaries that we have the following formulae
\begin{equation}\label{eq:exppai}
\begin{split}
\inn{k,u^n_{ij}} & = \de_{ij} \cd q^{j -n/2} \\
 \inn{e,u^n_{ij}} &= \de_{i,j - 1} \cd q^{\frac{1 - n}{2}} \sqrt{ \inn{n - j +1} \inn{j} } \\
\inn{f,u^n_{ij}} & = \de_{i,j + 1} \cd q^{\frac{1-n}{2}} \sqrt{ \inn{n-j} \inn{j+1} } 
\end{split}
\end{equation}
describing the pairing between the entries of the unitaries and the generators for $\C U_q(\G{su}(2))$. For $q = 1$, we have the same formulae for $\inn{e, u^n_{ij}}$ and $\inn{f,u^n_{ij}}$ but we moreover have that $\inn{h,u^n_{ij}} = \de_{ij} \cd (2j - n)$.\\

The left multiplication of the generators of $\C O(SU_q(2))$ on the entries of the irreducible unitary corepresentations $u^n$ are computed explicitly here below, using the convention that $u_{ij}^n:=0$ if $n < 0$ or $(i,j) \notin \{0,1,\ldots,n\}^2$: 
\begin{equation}\label{eq:leftmult}
\begin{split}
a^* \cd u^n_{ij} & = q^{i + j} \tfrac{ \sqrt{\inn{n - i + 1} \inn{n-j+1}}}{ \inn{n+1}} \cd u^{n+1}_{ij}
+ \tfrac{\sqrt{\inn{i} \inn{j}}}{\inn{n+1}} \cd u^{n-1}_{i-1,j-1} \\
b^* \cd u^n_{ij} & = q^j \tfrac{ \sqrt{\inn{i+1}\inn{n-j + 1}}}{\inn{n+1}} \cd u^{n+1}_{i+1,j}
- q^{i+1} \tfrac{ \sqrt{\inn{n-i} \inn{j}}}{\inn{n+1}} \cd u^{n-1}_{i,j-1} \\
a \cd u^n_{ij} & = \tfrac{ \sqrt{\inn{i + 1} \inn{j+1}}}{ \inn{n+1}} \cd u^{n+1}_{i+1,j+1}
+ q^{i+j+2}\tfrac{\sqrt{\inn{n-i} \inn{n-j}}}{\inn{n+1}} \cd u^{n-1}_{ij} \\
b \cd u^n_{ij} & = 
- q^{i-1}\tfrac{\sqrt{\inn{j+1} \inn{n-i +1}}}{\inn{n+1}} \cd u^{n+1}_{i,j+1}
+ q^j \tfrac{ \sqrt{\inn{n -j} \inn{i}}}{ \inn{n+1}} \cd u^{n-1}_{i-1,j} .
\end{split}
\end{equation}
These formulae can be derived from the $q$-Clebsch-Gordan coefficients and we refer the reader to \cite[Section 3]{DLSSV:DOS} and \cite[Chapter 3.4]{KlSc:QGR} for more information on these matters.\\

The Haar state on the $C^*$-completion $h \colon C(SU_q(2)) \to \cc$ is determined by the identities
\[
h(1) =1 \quad \T{and} \quad  h(u^n_{ij}) = 0,
\]
for all $n \in \nn$ and all $i,j \in \{0,1,\ldots,n\}$, see \cite[Chapter 4, Equation (50)]{KlSc:QGR}.  The Haar state is a twisted trace on $\C O(SU_q(2))$ with respect to the algebra automorphism $\nu \colon \C O(SU_q(2)) \to \C O(SU_q(2))$, which on the matrix units is given by
\begin{align}\label{eq:modular-function}
\nu(u^n_{ij}) =  q^{2(n - i - j)}   \cd u^n_{ij}.
\end{align}
That is,  we have that 
\begin{align}\label{eq:modular}
h(x y) = h(\nu(y) x)
\end{align}
for all $x,y \in \C O(SU_q(2))$, see \cite[Chapter 4, Proposition 15]{KlSc:QGR}. The Haar state is faithful and we let $L^2(SU_q(2))$ denote the Hilbert space completion of the $C^*$-algebra $C(SU_q(2))$ with respect to the induced inner product
\[
\inn{x,y} := h(x^* y), \quad x,y \in C(SU_q(2)) .
\]
The corresponding GNS-representation is denoted by
\[
\rho \colon C(SU_q(2)) \to \B B( L^2(SU_q(2))) .
\]
The entries from the irreducible unitary corepresentations $u^n_{ij}$ yield an orthogonal basis for $L^2(SU_q(2))$. Their norms are determined by
\begin{equation}\label{eq:haarmatrixI}
\inn{u^n_{ij},u^n_{ij}} = h( (u^n_{ij})^* u^n_{ij}) = \frac{q^{2(n-i)}}{\inn{n + 1} } ,
\end{equation}
see \cite[Chapter 4, Theorem 17]{KlSc:QGR}. It is also convenient to record that
\begin{equation}\label{eq:haarmatrixII}
\inn{ (u^n_{ij})^* , (u^n_{ij})^*} = h( u^n_{ij} (u^n_{ij})^*) = \frac{q^{2j}}{\inn{n + 1} } .
\end{equation}
Finally, we remark that for each $\eta \in \C U_q(\G{su}(2))$ one has $h\circ \de_\eta=h\circ \pa_\eta=\eta(1)\cdot h$ which follows directly from the bi-invariance of the Haar state.


\subsection{The D\k{a}browski-Sitarz spectral triple}
In the previous section we saw that the dual pairing gives rise to an action of $\C U_q(\G{su}(2))$ on  $\C O(SU_q(2))$ by linear endomorphisms, so in particular we obtain three linear maps
\[
\pa_e,  \pa_f, \pa_k \colon \C O(SU_q(2)) \to \C O(SU_q(2)) 
\]
from the generators $e,f,k \in \C U_q(\G{su}(2))$ of the quantum enveloping algebra. Using the pairing between  $\C U_q(\G{su}(2))$ and  $\C O(SU_q(2))$ together with the formulas  \eqref{eq:envhopf} one sees that $\pa_k$ is an algebra automorphism  and that $\pa_e$ and $\pa_f$ are twisted derivations, in the sense that
\begin{equation}\label{eq:derder}
\begin{split}
\pa_e(x \cd y) & = \pa_e(x) \pa_k(y) + \pa_k^{-1}(x) \pa_e(y), \\
\pa_f(x \cd y) & = \pa_f(x) \pa_k(y) + \pa_k^{-1}(x) \pa_f(y)
\end{split}
\end{equation}
for all $x,y \in \C O(SU_q(2))$; i.e.~the twist is determined by the algebra automorphism $\pa_k$. The behaviour of our three operations with respect to the involution on $\C O(SU_q(2))$ is also determined by \eqref{eq:envhopf} and the fact that we have a dual pairing of Hopf $*$-algebras:
\begin{equation}\label{eq:derstar}
\pa_e(x^*) = - q^{-1} \cd \pa_f(x)^* \, \, , \, \, \, \pa_f(x^*) = -q \pa_e(x)^* \, \, \T{and} \, \, \, \pa_k(x^*) = \pa_k^{-1}(x)^* 
\end{equation}
for all $x \in \C O(SU_q(2))$. For $q = 1$ we emphasise that our conventions imply that $\pa_k = \pa_1 = \T{id} \colon \C O(SU(2)) \to \C O(SU(2))$. In this case, both $\pa_e$ and $\pa_f$ are simply derivations on $\C O(SU(2))$, but we also have a third interesting derivation namely $\pa_h \colon \C O(SU(2)) \to \C O(SU(2))$ coming from the third generator $h \in \C U(\G{su}(2))$. The interaction between $\pa_h$ and the involution is encoded by the formula $\pa_h(x^*) = - \pa_h(x)^*$. \\
It is convenient to specify the explicit formulae
\begin{align}\label{eq:derexp}
 \pa_k(a) &= q^{1/2} a &   \pa_e(a) &= b^* & \pa_f(a) &= 0  \notag \\
 \pa_k(a^*) &= q^{-1/2} a^* &  \pa_e(a^*) &= 0 & \pa_f(a^*) &= -q b\\
 \pa_k(b) &= q^{1/2} b &   \pa_e(b) &= -q^{-1} a^* & \pa_f(b)&=0 \notag\\
\pa_k(b^*) &= q^{-1/2} b^* &   \pa_e(b^*)&=0& \pa_f(b^*) &= a \notag 
\end{align}
explaining the behaviour of the algebra automorphism $\pa_k$ and the two twisted derivations $\pa_e$ and $\pa_f$ on the generators 
for the coordinate algebra $\C O(SU_q(2))$. For $q = 1$, our extra derivation $\pa_h \colon \C O(SU(2)) \to \C O(SU(2))$ is given explicitly on the generators by
\[
\pa_h(a) = a \, \, , \, \, \, \pa_h(b) = b \, \, , \, \, \, \pa_h(a^*) = - a^* \, \, \, \T{and} \, \, \, \pa_h(b^*) = - b^* .
\]
For each $n \in \zz$, we let $\C A_n \su \C O(SU_q(2))$ denote the $n$'th spectral subspace coming from the strongly continuous circle action $\si_L \colon S^1 \ti \C O(SU_q(2)) \to \C O(SU_q(2))$ determined on the generators by $(z,a) \mapsto z \cd a$ and $(z,b) \mapsto z \cd b$. Thus, we let
\[
\C A_n := \big\{ x \in \C O(SU_q(2)) \mid \si_L(z,x) = z^n \cd x \, \, , \, \, \, \forall z \in S^1 \big\} .
\]
In particular, we have the fixed point algebra $\C A_0 \su \C O(SU_q(2))$, which is referred to as the \emph{coordinate algebra} for the standard \emph{Podle\'s sphere} and we apply the notation
\[
\C O(S_q^2) := \C A_0 = \big\{ x \in \C O(SU_q(2)) \mid \si_L(z,x) = x \, \, , \, \, \, \forall z \in S^1 \big\} .
\] 
The \emph{standard Podle\'s sphere} is defined as the $C^*$-completion of $\C O(S_q^2)$ using the $C^*$-norm inherited from $C(SU_q(2))$ and we apply the notation $C(S_q^2) \su C(SU_q(2))$ for this unital $C^*$-algebra. As the name suggests, the standard Podle\'s sphere was introduced by Podle\'s in \cite{Pod:QS} together with a whole range of ``non-standard'' Podle\'s spheres which we are not considering here.
We shall also refer to the (standard) Podle{\'s} sphere as the \emph{quantised 2-sphere} or the \emph{$q$-deformed 2-sphere}, whenever linguistically convenient. \\
The coordinate algebra $\C O(S_q^2)$ is generated by the elements
\[
A := bb^* \quad B = ab^* \quad B^* = ba^*
\]
and the following set 
\[
\big\{ A^i B^j, A^i (B^*)^k \mid i,j \in \nn_0 \, , \, \, k \in \nn \big\}
\]
constitutes a vector space basis for this coordinate algebra, see \cite[Theorem 1.2]{Wor:UAC}. For $q \neq 1$, it therefore follows that $\pa_k$ fixes $\C O(S_q^2)$, and another application of   \cite[Theorem 1.2]{Wor:UAC}   shows that this is actually an alternative description of the coordinate algebra:
\[
\C O(S_q^2) = \big\{ x \in \C O(SU_q(2)) \mid \pa_k(x) = x \big\} .
\]
In terms of the irreducible unitary corepresentations $\{u^n\}_{n = 0}^\infty$, the coordinate algebra $\C O(S_q^2)$ can be described as
\begin{equation}\label{eq:midcol}
\T{span}\big\{ u^{2m}_{im} \mid m \in \nn_0 \, , \,\, i \in \{0,1,\ldots,2m\} \big\} .
\end{equation}
Indeed, one may use the description of the pairing with $k \in \C U_q(\G{su}(2))$ from \eqref{eq:exppai} to obtain the formula $\pa_k(u^n_{ij}) = q^{j -n/2} u^n_{ij}$ for all $n \in \nn_0$ and all $i,j \in \{0,1,\ldots,n\}$. \\

We let $H_n$ denote the Hilbert space completion of $\C A_n$ with respect to the inner product inherited from $L^2(SU_q(2))$ and we put $H_+ := H_1$ and $H_- := H_{-1}$. We consider the direct sum $H_+ \op H_-$ as a $\zz/2\zz$-graded Hilbert space with grading operator $\ga = \ma{cc}{1 & 0 \\ 0 & -1}$. The restriction of the GNS-representation $\rho \colon C(SU_q(2)) \to \B B\big(L^2(SU_q(2))\big)$ to the standard Podle\'s sphere then provides us with an injective, even $*$-homomorphism
\[
\pi \colon C(S_q^2) \to \B B(H_+ \op H_-) \quad \T{ given by } \quad \pi(x) := \ma{cc}{\rho(x)\vert_{H_+} & 0 \\ 0 & \rho(x)\vert_{H_{-}}} .
\]
The definition of the circle action together with the identities in \eqref{eq:derder} and \eqref{eq:derexp} entail that
\begin{equation}\label{eq:dercirc}
\si_L(z, \pa_e(x)) = z^{-2} \pa_e\big(  \si_L(z,x) \big) \quad \T{and} \quad
\si_L(z, \pa_f(x)) = z^2 \pa_f\big( \si_L(z,x) \big)
\end{equation}
for all $z \in S^1$ and $x \in \C O(SU_q(2))$. In particular, we obtain two unbounded operators
\[
\C E \colon \C A_1 \to H_- \quad \T{and} \quad \C F \colon \C A_{-1} \to H_+
\]
agreeing with the restrictions $\pa_e \colon \C A_1 \to \C A_{-1}$ and $\pa_f \colon \C A_{-1} \to \C A_1$ followed by the relevant inclusions. An application of the identities
\[
h(\pa_e(x) )  = 0 = h(\pa_f(x)) \quad x \in \C O(SU_q(2))
\]
together with the identities in \eqref{eq:derder} and \eqref{eq:derstar} shows that $\C F \su \C E^*$ and $\C E \su \C F^*$. We let
\[
E \colon \T{Dom}(E) \to L^2(SU_q(2)) \quad \T{and} \quad F \colon \T{Dom}(F) \to L^2(SU_q(2)) 
\]
denote the closures of $\C E$ and $\C F$, respectively.  The $q$-deformed \emph{Dirac operator} $D_q \colon \T{Dom}(D_q) \to H_+ \op H_-$ is the odd unbounded operator given by
\[
D_q := \ma{cc}{0 & F \\ E & 0} \quad \T{with} \quad \T{Dom}(D_q) := \T{Dom}(E) \op \T{Dom}(F),
\]
and the main result from \cite{DaSi:DSP} is the following:
\begin{thm}\cite[Theorem 8]{DaSi:DSP}\label{t:spectrip}
The triple $(\C O(S_q^2), H_+ \op H_-, D_q)$ is an even spectral triple.
\end{thm}
%

The commutator with the Dirac operator $D_q \colon \T{Dom}(D_q) \to H_+ \op H_-$ induces a $*$-derivation $\pa \colon \C O(S_q^2) \to \B B(H_+ \op H_-)$, and since $D_{q}$ is odd the the $2\times 2$-matrix representing $\pa(x)$ is off-diagonal, and we denote it as follows: 
\[
\pa(x) = \ma{cc}{0 & \pa_2(x) \\ \pa_1(x) & 0}.
\]
The $*$-derivation $\pa \colon \C O(S_q^2) \to \B B(H_+ \op H_-)$ is closable since the Dirac operator $D_q$ is selfadjoint and therefore in particular closed. For $x\in \C O(S_q^2)$, an application of the twisted Leibniz rule \eqref{eq:derder} yields that $\pa_1(x)=q^{1/2}\rho(\pa_e(x))\vert_{H_+}$ and $\pa_2(x)=q^{-1/2}\rho(\pa_f(x))\vert_{H_-}$,  and in the sequel we will therefore often think of $\pa_1$ and  $\pa_2$ as the derivations 
\[
q^{1/2}\pa_e, q^{-1/2}\pa_f \colon \C O(S_q^2)\to \C O(SU_q(2))
\]
rather than their representations as bounded operators.

\begin{remark}
In \cite{DaSi:DSP} the even spectral triple on $\C O(S_q^2)$ is also equipped with an extra antilinear operator $J \colon H_+ \op H_- \to H_+ \op H_-$ (the reality operator). It is then verified that the even spectral triple is in fact both real and $\C U_q(\G{su}(2))$-equivariant and that these properties determine the spectral triple up to a non-trivial scalar $z \in \cc \sem \{0\}$. The description of the Dirac operator in terms of the dual pairing of Hopf $*$-algebras, which we are using here, can be found in \cite[Section 3]{NeTu:LFQ}.
\end{remark}

\begin{remark}\label{r:dirac}
In the case where $q = 1$, it can be proved that the direct sum of Hilbert spaces $H_+ \op H_-$ agrees with the $L^2$-sections of the spinor bundle $S_+ \op S_- \to S^2$ on the $2$-sphere. Letting $\Ga^\infty(S_+ \op S_-)$ denote the smooth sections of the spinor bundle it can moreover be verified that the unbounded selfadjoint operator $D_1 \colon \T{Dom}(D_1) \to H_+ \op H_-$ agrees with the closure of the Dirac operator $\C D : \Ga^\infty(S_+ \op S_-) \to \Ga^\infty(S_+ \op S_-)$ upon considering $\C D$ as an unbounded operator on $H_+ \op H_-$; see e.g.~\cite[Section 3.5]{Friedrich:Dirac}. For more information on the spin geometry of the $2$-sphere, we refer the reader to \cite[Chapter 9A]{GrVaFi:ENG}.
\end{remark}


\subsection{Compact quantum metric spaces}\label{subsec:cqms}
In this section we gather the necessary background material concerning compact quantum metric spaces. These are the natural noncommutative analogues of classical compact metric spaces, and were introduced by Rieffel \cite{Rie:MSS, Rie:GHD} around  the year 2000. The basic idea is that the  noncommutative counterpart to a classical metric is captured by a certain seminorm, the domain of which can be chosen in several ways, leading to slight variations of the same theory. Rieffel's original theory \cite{Rie:GHD} is formulated in the language of order unit spaces, but one may equally well take a $C^*$-algebraic setting as the point of departure \cite{Li:CQG, Li:GH-dist, Rie:MSS}. We will here present a generalisation of the $C^*$-algebraic setting and take concrete operator systems as our point of departure.\\

Let $X$ be a concrete operator system; thus, for our purposes, $X$ is a closed subspace of a specified unital $C^*$-algebra such that $X$ is stable under the adjoint operation and contains the unit. The operator system $X$ has a state space $\C S(X)$ consisting of all the positive linear functionals preserving the unit. 

\begin{dfn}\label{def:cqms}
A \emph{compact quantum metric space} is a concrete operator system $X$ equipped with a seminorm $L\colon X \to [0,\infty]$ satisfying the following:
\begin{itemize}
\item[(i)] One has $L(x)=0$ if and only if $x\in \mathbb{C}\cdot 1$.
\item[(ii)] The set $\T{Dom}(L):=\{x\in X \mid L(x)<\infty\}$ is dense in $X$ and $L$ satisfies that $L(x^*)=L(x)$ for all $x \in X$.
\item[(iii)] The function $\rho_L(\mu,\nu):=\sup\{|\mu(x)-\nu(x)| \mid L(x)\leq 1\}$ defines a metric on the state space $\C S(X)$ which metrises the weak$^*$-topology.
\end{itemize}
In this case, the seminorm $L$ is referred to as a \emph{Lip-norm} and the corresponding metric is referred to as the \emph{Monge-Kantorovi\v{c} metric}.
\end{dfn}
A compact quantum metric space $(X,L)$ has an associated order unit space $A:=\{x\in X_{\sa} \mid L(x)<\infty \}= X_{\T{sa}}\cap \T{Dom}(L)$, where both the order and the unit are inherited from the ambient unital $C^*$-algebra. In the setting of order unit spaces, the notion of a state also makes sense, and we note that the restriction map provides an identification $\C S(X)\cong \C S(A)$. Moreover, it is easy to verify that $(A,L\vert_A)$ is an \emph{order unit compact quantum metric space} in the sense of Rieffel; see \cite{Rie:MSA, Rie:MSS, Rie:GHD}.  The restriction of $L$ to $A$ defines a Monge-Kantorovi\v{c} metric on $\C S(A)$ by the obvious modification of the formula in (iii),
and the requirement that $L$ be invariant under the involution implies that the identification of state spaces $\C S(X) \cong \C S(A)$ becomes an isometry for the Monge-Kantorovi\v{c} metrics; for more details on these matters, see \cite[Section 2]{kaad-kyed}.\\


The canonical commutative example upon which the above definition is modelled, arises by considering a compact metric space $(M,d)$ and its associated $C^*$-algebra $C(M)$, which can be endowed with a seminorm by setting 
\[
L_d(f):=\sup\left\{\frac{|f(p)-f(q)|}{d(p,q)} \mid p,q\in M, p\neq q\right\}.
\]
Then $L_d(f)$ is finite exactly when $f$ is Lipschitz continuous, in which case $L_d(f)$ is the Lipschitz constant.  By results of Kantorovi\v{c} and Rubin\v{s}te\u{\i}n \cite{KaRu:FSE,KaRu:OSC}, one has that  $\rho_{L_d}$ metrises the weak$^*$-topology on $\C S(C(M))$, so that $(C(M),L_d)$ is indeed a compact quantum metric space, and  that the restriction of $\rho_{L_d}$ to $M\su \C S(C(M))$ agrees with the  metric $d$.  \\

At first glance, it might seem like a difficult task to verify that the function $\rho_L$ metrises the weak$^*$-topology, but actually this can be reduced to a compactness question as the following theorem shows.

\begin{thm}[{\cite[Theorem 1.8]{Rie:MSA}}]\label{t:totallybdd}
Let $X$ be a concrete operator system and $L\colon X\to [0,\infty]$ a seminorm satisfying (i) and (ii) from Definition \ref{def:cqms}. Then $(X,L)$ is a compact quantum metric space if and only if the image of the Lip-unit ball $\{x\in X\mid L(x)\leq 1\}$ under the quotient map $X\to X/\mathbb{C}\cdot 1$ is totally bounded for the quotient norm.
\end{thm}

Note, in particular, that this implies that if $X$ is finite dimensional, then any seminorm satisfying (i) and (ii) from Definition \ref{def:cqms} automatically provides $X$ with a quantum metric structure; we will use this fact repeatedly without further reference in the sequel. \\

One of the many pleasant features of the theory of compact quantum metric spaces, is that it allows for  a noncommutative analogue of the classical Gromov-Hausdorff distance between compact metric spaces \cite{gromov-groups-of-polynomial-growth-and-expanding-maps, edwards-GH-paper},  which we now recall. If  $(X,L_X)$ and $(Y,L_Y)$  are two compact quantum metric spaces, denote by $A$ and $B$ the associated order unit compact quantum metric spaces $A:=X_{\T{sa}} \cap \T{Dom}(L_X)$ and $B:=Y_{\T{sa}}\cap \T{Dom}(L_Y)$. A finite (real) seminorm $L\colon A\oplus B \to[0,\infty)$ is called \emph{admissible} if $(A\oplus B, L)$ is an order unit compact quantum metric space and if the associated quotient seminorms on $A$ and $B$ agree with $L_X\vert_A$ and $L_Y\vert_B$, respectively.  
For any such $L$, one obtains isometric embeddings of state spaces: 
\begin{align*}
(\C S(X), \rho_{L_X}) &\cong (\C S(A), \rho_{L_X\vert_A}) \hookrightarrow (\C S (A\oplus B), \rho_L) \\
 (\C S(Y), \rho_{L_Y}) &\cong (\C S(B), \rho_{L_Y\vert_B}) \hookrightarrow (\C S (A\oplus B), \rho_L) 
\end{align*}
and hence one can consider the \emph{Hausdorff distance}  $\dist_{\T{H}}^{\rho_L}(\C S(X), \C S(Y))$ (see  \cite{Hausdorff-grundzuge}). The \emph{quantum Gromov-Hausdorff distance} between $(X,L_X)$ and $(Y,L_Y)$ is then defined as
\[
\dist_{\T{Q}} ((X,L_X); (Y,L_Y)):= \inf \left\{ \dist_{\T{H}}^{\rho_L}(\C S(X), \C S(Y)) \mid L\colon A\oplus B\to [0,\infty) \T{ admissible}   \right\} .
\]
We remark that this is simply a rephrasing of Rieffels original definition from \cite{Rie:GHD}, where everything is formulated in terms of order unit spaces, in the sense that $\dist_{\T{Q}}((X,L_X); (Y,L_Y)):=\dist_{\T{Q}}((A,L_X\vert_A); (B,L_Y\vert_B))$.
In \cite{Rie:GHD}, Rieffel  showed that $\dist_{\T{Q}} $ is symmetric and satisfies the triangle inequality, and that distance zero is equivalent to the existence of a Lip-norm preserving isomorphism of order unit spaces between the (completions of the) quantum metric spaces in question, see \cite{Rie:GHD} for details on this.  
Over the past 20 years, several refinements of the quantum Gromov-Hausdorff distance have been proposed \cite{Ker:MQG, Lat:QGH, Li:CQG} for which distance zero actually implies Lip-isometric isomorphism at the $C^*$-level. A very successful such refinement is Latr{\'e}moli{\`e}re's notion of \emph{ quantum propinquity}, which has been developed in a series of influential papers \cite{Lat:AQQ, Lat:BLD, Lat:DGH, Lat:QGH}.  In the present text, however, we will only consider Rieffel's original notion, and shall therefore not elaborate further on the quantum propinquity.\\

Rieffel's definition of compact quantum metric spaces is drawing inspiration from Connes' noncommutative geometry, and the latter is therefore, not surprisingly, a source of interesting examples of compact quantum metric spaces, \cite[Chapter 6]{Con:NCG} and \cite{Con:CFH}. Concretely, if $(\C A, H, D)$ is a unital spectral triple with $\C A$ sitting as a dense unital $*$-subalgebra of the unital $C^*$-algebra $A \su \B B(H)$, then one obtains a seminorm $L_D\colon A\to [0,\infty]$ by setting
\begin{equation}\label{eq:semispec}
L_D(a):= \fork{ccc}{ \| \ov{[D,a]} \| & \T{for} & a \in \C A \\ \infty & \T{for} & a \notin \C A} .
\end{equation}

The main result of \cite{AgKa:PSM} is the following:


\begin{thm}\label{t:specmetpod}
For each $q \in (0,1)$, the seminorm $L_{D_q}$ arising from the D\k{a}browski-Sitarz spectral triple $(\C O(S_q^2),H_+ \op H_-,D_q)$ turns $C(S_q^2)$ into a compact quantum metric space. 
\end{thm}

\begin{remark}\label{r:lipspec}
For a unital spectral triple $(\C A,H,D)$ there is also a \emph{maximal} domain for the associated seminorm called the \emph{Lipschitz algebra} and denoted by $A^{\Lip} \su A$. This dense unital $*$-subalgebra consists of all elements  $a \in A$ satisfying that $a(\T{Dom}(D)) \su \T{Dom}(D)$ and that $[D,a] \colon \T{Dom}(D) \to H$ extends to a bounded operator on $H$. The corresponding seminorm is denoted by $L_D^{\T{max}}$. In general, the Lip-algebra does not agree with the domain of the closure of the derivation $\pa \colon \C A \to \B B(H)$ given by $\pa(a) := \ov{[D,a]}$. As a consequence, we do not know whether it holds that, if $(A,L_D)$ a compact quantum metric space, then $(A,L_D^{\T{max}})$ is a compact quantum metric space. The converse can, however, be verified immediately by an application of Theorem \ref{t:totallybdd}. It is an interesting problem to investigate the relationship between $(A,L_D)$ and $(A,L_D^{\T{max}})$ -- in particular whether the quantum Gromov-Hausdorff distance between the two is equal to zero. 
In \cite[Theorem 8.3]{AgKa:PSM} it is actually proved (for $q \in (0,1)$)  that $(C(S_q^2),L_{D_q}^{\T{max}})$ is a compact quantum metric space and in \cite[Theorem A]{AKK:DistZero} we show that the quantum Gromov-Hausdorff distance between $(C(S_q^2),L_{D_q})$ and $(C(S_q^2),L_{D_q}^{\T{max}})$ is indeed zero. Thus, the main convergence result in Theorem \ref{mainthm:A} holds true also when the seminorm $L_{D_q}$ is replaced by its maximal counterpart $L_{D_q}^{\T{max}}$.
\end{remark}

\begin{remark}
In the case where $q = 1$ we have discussed earlier in Remark \ref{r:dirac} that the unbounded selfadjoint operator $D_1 \colon \T{Dom}(D_1) \to H_+ \op H_-$ agrees with the closure of the Dirac operator $\C D \colon \Ga^\infty(S_+ \op S_-) \to H_+ \op H_-$ coming from the spin geometry of the $2$-sphere. We therefore know from \cite[Proposition 1]{Con:CFH} that $L_{D_1}(f)$ agrees with the Lipschitz constant of $f$ associated with the round metric on the $2$-sphere when $f \in \C O(S^2)$ (and outside of $\C O(S^2)$,    $L_{D_q}$ takes the value $\infty$, by definition). In particular, we obtain that the Monge-Kantorovi\v{c} metric coming from the seminorm $L_{D_1} \colon C(S^2) \to [0,\infty]$ agrees with the Monge-Kantorovi\v{c} metric coming from the round metric on the $2$-sphere. We may thus conclude that $(C(S^2),L_{D_1})$ is a compact quantum metric space as well.
\end{remark}

In the recent papers \cite{walter-connes:truncations, walter:GH-convergence} the notion of a unital spectral triple was extended by replacing the $C^*$-algebra by an operator system. For such an operator system spectral triple $(\C X,H,D)$ one may again form a seminorm using the formula \eqref{eq:semispec} and it then makes sense to ask whether the data $(X,L_D)$ is a compact quantum metric space. We shall see examples of this phenomenon in our analysis of quantum fuzzy spheres here below, see Section \ref{ss:quafuz}.



\section{Quantum fuzzy spheres}\label{sec:quantum-fuzzy-spheres}
In this section we introduce the key ingredients needed to prove the convergence of the Podle\'s spheres towards the classical round sphere. Our proof of convergence proceeds via a finite dimensional approximation procedure involving a quantum version of the fuzzy spheres. We are going to consider these quantum fuzzy spheres as finite dimensional operator system spectral triples sitting inside the D\k{a}browski-Sitarz spectral triple for the corresponding Podle\'s sphere. The operation which links the quantum fuzzy spheres to the Podle\'s sphere is then provided by a quantum analogue of the classical Berezin transform. We are now going to describe all these ingredients and once this is carried out the present section culminates with a proof of the Lip-norm contractibility of the quantum Berezin transform.
Let us once and for all fix an $N \in \nn$ and a deformation parameter $q \in (0,1]$.

\subsection{The quantum Berezin transform}
Our first ingredient is a quantum analogue of the Berezin transform. This is going to be a positive unital map $\be_N \colon C(S_q^2) \to C(S_q^2)$ which has a finite dimensional image. The aim of this section is to introduce the Berezin transform and compute its image.
We define the state $h_N \colon { C(S_q^2)} \to \cc$ by the formula
\begin{align}\label{eq:statedef}
h_N(x) := \inn{N+1} \cd h\big( (a^*)^N x a^N \big) .
\end{align}
Remark that $h_N$ is indeed unital since $(a^*)^N = u^N_{00}$ and therefore
\[
h_N(1) = \inn{N+1} \cd h\big( u^N_{00} \cd (u^N_{00})^*\big) = 1
\]
by the identity in \eqref{eq:haarmatrixII}. In the definition here below, we let $\De \colon C(S_q^2) \to C(SU_q(2)) \ot_{\T{min}} C(S_q^2)$ denote the left coaction of quantum $SU(2)$ on the Podle\'s sphere. This coaction comes from the restriction of the coproduct $\De \colon C(SU_q(2)) \to C(SU_q(2)) \ot_{\T{min}} C(SU_q(2))$ to the Podle\'s sphere $C(S_q^2) \su C(SU_q(2))$. 

\begin{dfn} The \emph{quantum Berezin transform} in degree $N \in \nn$ is the positive unital map 
$
\be_N \colon C(S_q^2) \to C(S_q^2)$ given by $\be_N(x) := (1 \ot h_N) \De(x) $.

\end{dfn}

We notice that $\be_N$ would a priori take values in $C(SU_q(2))$, but the following lemma shows that $\be_N \colon C(S_q^2) \to C(SU_q(2))$ does indeed factorise through $C(S_q^2)$. Recall to this end that the elements $u^{2m}_{im}$ for $m \in \nn_0$ and $i \in \{0,1,\ldots,2m\}$ form a vector space basis for the coordinate algebra $\C O(S_q^2)$. 

\begin{lemma}\label{l:altI}
For every $m \in \nn_0$ and $i \in \{0,1,2,\ldots,2m\}$ we have the formula
\[
\be_N(u^{2m}_{im} ) = u^{2m}_{im} \cd h_N(u^{2m}_{mm}) .
\]
\end{lemma}
\begin{proof}
Let $m \in \nn_0$ and $i \in \{0,1,2,\ldots,2m\}$ be given. We compute that
\[
\begin{split}
\be_N( u^{2m}_{im}) 
& = (1 \ot h_N) \De(u^{2m}_{im}) = \sum_{l = 0}^{2m} u^{2m}_{il} \cd h_N( u^{2m}_{lm}) \\
& = \sum_{l = 0}^{2m} u^{2m}_{il} \cd h\big( (a^* )^N u^{2m}_{lm} a^N \big) \inn{N+1}
= u^{2m}_{im} \cd h_N(u^{2m}_{mm}) ,
\end{split}
\]
where the last identity follows since the Haar state $h \colon \C O(SU_q(2)) \to \cc$ is invariant under the algebra automorphism $\nu \colon \C O(SU_q(2)) \to \C O(SU_q(2))$.
\end{proof}
\begin{remark}
The (unpublished) paper \cite{Sain:thesis} also provides a  version of the Berezin transform for certain quantum homogeneous spaces, but the setting is that of Kac type quantum groups and the constructions therefore do not directly apply to $SU_q(2)$ and $S_q^2$ when $q\neq 1$.
\end{remark}

For classical spaces, the Berezin transform is a well studied object (cf. \cite{Sch:BTQ} and references therein), and as we shall now see, our quantum Berezin transform exactly recovers the classical construction when $q=1$. 
So let us for a little while assume that $q = 1$. Recall first that the irreducible corepresentation $u^N\in \mathbb{M}_{N+1}(C(SU(2)))$ is actually the same as an irreducible representation $u^N\colon SU(2) \to U(\mathbb{C}^{N+1})$. Let $\operatorname{Tr}$ denote the trace on $\mathbb{M}_{N+1}(\mathbb{C})$ without normalisation so that $\operatorname{Tr}(1) = N + 1$. We choose $P\in \mathbb{M}_{N+1}(\mathbb{C})$ to be the rank one projection with $P_{00} = 1$ and all other entries equal to zero. Viewing $C(S^2)$ as the fixpoint algebra $C(SU(2))^{S^1}$, the classical Berezin transform $b_N \colon C(S^2) \to C(S^2)$ is given by  (cf.~\cite[Section 3.3.1]{walter:GH-convergence} or \cite[Section 2]{Rie:MSG})
\[
b_N(f)(g):= (N+1)\int_{SU(2)} f(g \cd x^{-1})H_N(x) d\lambda(x), \quad { g \in SU(2)}, 
\]
where { $H_N$} denotes the density $H_N(x):=\operatorname{Tr}\big(P u^N(x)P u^N(x)^*\big)$ and $\lambda$ is the Haar probability measure on $SU(2)$. We remark that the trace property implies that $H_N(x)=H_N(x^{-1})$ which together with the unimodularity of $SU(2)$ gives 
\[
b_N(f)(g)=(N+1)\int_{SU(2)} f(g \cd x)H_N(x) d\lambda(x), \quad { g \in SU(2)}.
\]
An element $x$ in $SU(2)$ is (in the first irreducible representation) given by a complex matrix $\begin{pmatrix} \bar{z_1}  &  -z_2 \\ \bar{z_2} & {z_1}\\ \end{pmatrix}$ and the functions $a$ and $b$ then correspond to mapping the matrix to $z_1$ and $z_2$, respectively. Recall also that we have chosen the irreducible corepresentations $u^N$ so that $u^N_{00}=(a^*)^N$. A direct computation now shows that
\[
\operatorname{Tr}\big(P u^N(x)P u^N(x)^*\big)=\bar{z_1}^Nz_1^N =(a^*)^N(x) a^N(x).
\]
Since the Haar state $h$ at $q=1$ is given by integration against $\lambda$, we obtain from this that
\begin{align*}
(N+1)\int_{SU(2)} f(g \cd x)H_N(x) d\lambda(x) &= (N+1)\int_{SU(2)} \Delta(f)(g,x)(a^*)^N(x) a^N(x) d\lambda(x) \\
&=\inn{N+1} \cd h\big( \Delta(f)(g,-)(a^*)^Na^N \big)\\
&=(1\otimes h_N)(\Delta(f))(g) ,
\end{align*}
whenever $g$ belongs to $SU(2)$. This shows that our quantum Berezin transform $\be_N$ agrees with the classical Berezin transform $b_N$ when $q=1$.\\

We now return to the more general setting where the deformation parameter $q$ belongs to $(0,1]$. We apply the convention that $u_{ij}^n = 0$ whenever $n < 0$ or $n \in \nn_0$ and $(i,j) \notin \{0,1,\ldots,n\}^2$. 

\begin{lemma}\label{l:altII}
The image of $\be_N \colon C(S_q^2) \to C(S_q^2)$ agrees with the linear span:
\[
\T{span}_{\cc}\big\{ u^{2m}_{im} \mid m \in \{0,1,\ldots,N\} \, , \, \, i \in \{0,1,\ldots,2m\} \big\} .
\]
In particular, we have that $\T{Im}(\be_N) \su \C O(S_q^2)$ and that $\T{Dim}( \T{Im}(\be_N)) = (N + 1)^2$. 
\end{lemma}
\begin{proof}
Let first $n \in \nn_0$ and $i,j \in \{0,1,\ldots,n\}$ be given. Applying the formulae from \eqref{eq:leftmult} we obtain that
\[
\begin{split}
a^N \cd u^n_{ij} & =  \sum_{k = 0}^N \la_{n,i,j}(k) \cd u_{i+k,j+k}^{n + 2k - N} \quad \T{and} \\
(a^*)^N \cd u^n_{ij} & = \sum_{k = 0}^N \mu_{n,i,j}(k) \cd u_{i -k,j-k}^{n-2k + N} ,
\end{split}
\]
where all the coefficients appearing are strictly positive. In particular, we may find strictly positive coefficients such that
\[
a^N (a^*)^N \cd u^n_{ij} = \sum_{k = -N}^N \al_{n,i,j}(k) \cd u^{n + 2k}_{i + k,j+k} .
\]
Let now $m \in \nn_0$ be given. Since $h(u^n_{ij}) = 0$ for all $n > 0$ and $h(u^0_{00}) = 1$ we obtain that
\[
\begin{split}
h_N(u^{2m}_{m,m}) 
& = h\big( (a^*)^N u^{2m}_{m,m} a^N \big) \cd \inn{N+1}
= h\big( a^N (a^*)^N \cd u^{2m}_{m,m}\big) q^{-2N} \cd \inn{N+1} \\
& = \sum_{k = -N}^N \al_{2m,m,m}(k) \cd h( u^{2m + 2k}_{m + k,m+k} ) q^{-2N} \cd \inn{N+1} \\
& = \fork{ccc}{
0 & \T{for} & m > N \\
\al_{2m,m,m}(-m) q^{-2N} \cd \inn{N+1} & \T{for} & m \leq N
} .
\end{split}
\]
An application of Lemma \ref{l:altI} then proves the result of the present lemma.
\end{proof}

\subsection{Quantum fuzzy spheres}\label{ss:quafuz}
Our second ingredient is a quantum analogue of the fuzzy spheres, which we will introduce in this section, and afterwards equip each of them with an operator system spectral triple. These operator system spectral triples provide each of the quantum fuzzy spheres with the structure of a compact quantum metric space. Moreover, we are going to link the quantum fuzzy spheres to the Podle\'s spheres by showing that the image of the Berezin transform in degree $N$ agrees with the quantum fuzzy sphere in degree $N$, thus obtaining natural quantum analogues of classical results, see \cite{walter-connes:truncations, Rie:MSG, walter:GH-convergence}.

\begin{dfn}\label{d:quantumfuzz}
We define the {\it quantum fuzzy sphere} in degree $N \in \nn$ as the $\cc$-linear span
\[
F_q^N := \T{span}_\cc\big\{ A^i B^j, A^i (B^*)^j \mid i,j \in \nn_0 \, , \, \, i + j \leq N \big\} \su C(S_q^2) .
\]
\end{dfn}

We immediately remark that the vector space dimension of $F_q^N$ agrees with $(N + 1)^2$ which in turn is the dimension of the classical fuzzy sphere $M_{N+1}(\cc)$. Since $F_q^N \su C(S_q^2)$ is closed, unital and stable under the adjoint operation we may think of the quantum fuzzy sphere as a concrete operator system. Moreover, since $(\C O(S_q^2), H_+ \op H_-, D_q)$ is an even unital spectral triple we immediately obtain an even operator system spectral triple $( F_q^N, H_+ \op H_-, D_q)$ for the quantum fuzzy spheres. In particular, we may equip $F_q^N$ with the seminorm $L_{D_q} \colon F_q^N \to [0,\infty)$ defined by
\[
L_{D_q}(x) := \max\{ \| \pa_1(x) \| , \| \pa_2(x) \| \} .
\]
Thus, $L_{D_q}$ on the quantum fuzzy sphere is just the restriction of the seminorm on $C(S_q^2)$ arising from the D\k{a}browski-Sitarz spectral triple. Since we already know that $L_{D_q}$ is a Lip-norm on $C(S_q^2)$ we immediately obtain that $L_{D_q}$ is also a Lip-norm on $F_q^N$. We summarise this observation in a lemma:

\begin{lemma}
The pair $(F_q^N,L_{D_q})$ is a compact quantum metric space. 
\end{lemma}

We shall now see that the quantum fuzzy sphere in degree $N$ agrees with the image of the Berezin transform $\be_N \colon C(S_q^2) \to C(S_q^2)$.

\begin{lemma}\label{lem:fuzzy-equal-image}
It holds that $F_q^N = \T{Im}(\be_N)$.
\end{lemma}
\begin{proof}
Since the vector space dimension of $F_q^N$ agrees with the vector space dimension of $\T{Im}(\be_N)$ it suffices to show that $F_q^N \su \T{Im}(\be_N)$ (see Lemma \ref{l:altII}). Moreover, since $\be_N(x^*) = \be_N(x)^*$ we only need to show that $A^k B^l \in \T{Im}(\be_N)$ for all $k,l \in \nn_0$ with $k + l \leq N$. However, using that $A = b b^*$ and $B = a b^*$ we see from \eqref{eq:leftmult} that
\[
A^k B^l = (b b^*)^k (ab^*)^l \in \T{span}_{\cc}\big\{ u^n_{ij} \mid n \leq 2(k+l) \, , \, \, i,j \in \{0,1,\ldots,n\} \big\}.
\]
Moreover, since $A^k B^l \in \C O(S_q^2)$ we must in fact have that
\[
A^k B^l \in \T{span}_{\cc}\big\{ u^{2m}_{im} \mid m \leq k + l \, , \, \, i \in \{0,1,\ldots,2m \} \big\} ,
\]
see \eqref{eq:midcol}. Since $k + l \leq N$ we now obtain the result of the present lemma by applying Lemma \ref{l:altII}. 
\end{proof}

When $q=1$ the classical fuzzy sphere in degree $N$ is, by definition, given as $\mathbb{M}_{N+1}(\cc)$ and the classical Berezin transform agrees with the composition $b_N =  \sigma_N \circ \breve \sigma_N$, where $\sigma_N\colon \mathbb{M}_{N+1}(\cc) \to C(S^2)$ is the so-called covariant Berezin symbol and $\breve \sigma_N$ is its adjoint (see eg.~\cite[Section 3.3.1]{walter:GH-convergence} or \cite[Section 2]{Rie:MSG}). In the quantised setting we have only defined the composition thus leaving out a treatment of the covariant Berezin symbol. However, since the quantum Berezin transform at $q=1$ agrees with the classical Berezin transform we obtain that $F^N_1=\sigma_N\big(\mathbb{M}_{N+1}(\cc)\big)$. Using our seminorm $L_{D_1} \colon F^N_1 \to [0,\infty)$ we therefore also obtain a seminorm on the classical fuzzy sphere $\mathbb{M}_{N+1}(\cc)$ by deeming the covariant Berezin symbol to be a Lip-norm isometry. At least a priori, this seminorm is different from the one considered by Rieffel in \cite{Rie:MSG} and the one arising from the Grosse-Pre\v{s}najder Dirac operator considered in \cite{Bar:MFF, GrPr:Dirac-fuzzy, walter:GH-convergence}.

\subsection{Derivatives of the Berezin transform}\label{s:der}
Our aim in this section is to show that the quantum Berezin transform is a Lip-norm contraction and we are going to achieve this for elements in the coordinate algebra $\C O(S_q^2)$. In Proposition \ref{p:derVI} here below we shall thus see that $L_{D_q}(\be_N(x)) \leq L_{D_q}(x)$ for all $x \in \C O(S_q^2)$. It turns out that the derivation $\pa : \C O(S_q^2) \to \mathbb{M}_2\big( \C O(SU_q(2))\big)$ coming from the D\k{a}browski-Sitarz spectral triple does not have any good equivariance properties with respect to the Berezin transform and this makes the proof of the inequality $L_{D_q}(\be_N(x)) \leq L_{D_q}(x)$ a delicate matter. Our strategy is to conjugate the derivation $\pa$ with the fundamental corepresentation unitary $u \in \mathbb{M}_2( \C O(SU_q(2)))$ and thereby obtain an operation which is equivariant with respect to the Berezin transform. It is in fact possible to describe the conjugated derivation $u \pa u^*$ entirely in terms of the \emph{right action} of the quantum enveloping algebra, even though the derivation $\pa$ comes from the \emph{left action} of the quantum enveloping algebra.   To this end, recall that the right action on $\C O(SU_q(2))$ of an element $g \in \C U_q(\G{su}(2))$ is given by the linear endomorphism $\de_g \colon \C O(SU_q(2)) \to \C O(SU_q(2))$ defined by $\de_g(x) := (\inn{g,\cd} \ot 1) \De(x)$. Remark that $\de_g(x) \in \C O(S_q^2)$ for all $x \in \C O(S_q^2)$ since $\De(\C O(S_q^2)) \su \C O(SU_q(2)) \ot \C O(S_q^2)$.

\begin{lemma}\label{l:derI}
Let $g \in \C U_q(\G{su}(2))$. It holds that
\[
\de_g \be_N(x) = \be_N \de_g(x) 
\]
for all $x \in \C O(S_q^2)$.
\end{lemma}
\begin{proof}
Remark first that $\T{Im}(\be_N) \su \C O(S_q^2)$ so that it makes sense to look at the composition $\de_g \be_N$, see Lemma \ref{l:altII}. Then by coassociativity of the coproduct $\De \colon \C O(SU_q(2)) \to \C O(SU_q(2)) \ot \C O(SU_q(2))$ we have that
\[
\begin{split}
\de_g \be_N(x) 
& = (\inn{g,\cd} \ot 1)(\De \ot h_N)\De(x) 
= (\inn{g,\cd} \ot 1)(1 \ot 1 \ot h_N)(1 \ot \De)\De(x) \\
& = (1 \ot h_N)\De(\inn{g,\cd} \ot 1)\De(x) = \be_N \de_g(x)
\end{split}
\]
for all $x \in \C O(S_q^2)$. This proves the lemma.
\end{proof}

We are particularly interested in the three linear maps $\de_e, \de_f$ and $\de_k \colon \C O(SU_q(2)) \to \C O(SU_q(2))$. We record that $\de_k$ is an algebra automorphism whereas $\de_e$ and $\de_f$ are twisted derivations, meaning that
\begin{equation}\label{eq:delder} 
\begin{split}
\de_e(x \cd y) & = \de_e(x) \cd \de_k(y) + \de_k^{-1}(x) \cd \de_e(y) \quad \T{and} \\
\de_f(x \cd y) & = \de_f(x) \cd \de_k(y) + \de_k^{-1}(x) \cd \de_f(y)
\end{split}
\end{equation}
for all $x,y \in \C O(SU_q(2))$. The compatibility between our three operations and the involution on $\C O(SU_q(2))$ is described by the identities
\begin{equation}\label{eq:delstar}
\de_e(x^*) = -q^{-1} \cd \de_f(x)^* \, \, , \, \, \, \de_f(x^*) = -q \cd \de_e(x)^* \, \, \T{and} \, \, \, \de_k(x^*) = \de_k^{-1}(x)^* .
\end{equation}
 In the case where $q = 1$, we emphasise that our conventions imply that $\de_k$ agrees with the identity automorphism of $\C O(SU(2))$, and hence we obtain that $\de_e$ and $\de_f$ are derivations on $\C O(SU(2))$ in the usual sense of the word. However, for $q = 1$ we also have the interesting derivation $\de_h \colon \C O(SU(2)) \to \C O(SU(2))$ coming from the third generator $h \in \C U(\G{su}(2))$. This extra derivation relates to the adjoint operation via the formula $\de_h(x^*) = - \de_h(x)^*$.\\
It is convenient to record the formulae:

\begin{align}\label{eq:delexp}
\de_k(a) &= q^{1/2} \cd a&   \de_e(a) &= 0 & \de_f(a) &= -q \cd b  \notag \\
\de_k(a^*) &= q^{-1/2} \cd a^*&  \de_e(a^*) &= b^* &\de_f(a^*) &= 0 \\
\de_k(b) &= q^{-1/2} \cd b &   \de_e(b) &= -q^{-1} \cd a  &\de_f(b) &=0 \notag\\
 \de_k(b^*) &= q^{1/2} \cd b^*  &   \de_e(b^*)&=0& \de_f(b^*)& = a^*  \notag 
\end{align}
Moreover, for $q = 1$ we in addition have that
\[
\de_h(a) = a \, \, , \, \, \, \de_h(b) = -b \, \, , \, \, \, \de_h(a^*) = -a^* \, \,  \T{and } \, \, \de_h(b^*) = b^* . 
\]
We define the algebra automorphism $\tau := \de_k \pa_k \colon \C O(SU_q(2)) \to \C O(SU_q(2))$ and notice that it follows from the defining commutation relations in $\C O(SU_q(2))$ that
\begin{equation}\label{eq:taucommu}
b x = \tau(x) b \quad \T{and} \quad b^* x = \tau(x) b^* \quad \T{for all } x \in \C O(SU_q(2)) .
\end{equation}
Clearly, $\tau$ agrees with $\de_k$ when restricted to $\C O(S_q^2)$. In general, when $\te \colon \C O(SU_q(2)) \to \C O(SU_q(2))$ is an algebra automorphism, we shall apply the notation
\[
[y,x]_\te := y x - \te(x) y 
\]
for the twisted commutator between two elements $x$ and $y \in \C O(SU_q(2))$. With this notation we now obtain:

\begin{lemma}\label{l:deriII}
We have the identities
\[
[a^*,x]_{\de_k} = (1 - q^2) q^{1/2} b \pa_e(x) \quad \mbox{and} \quad
[a, x]_{\de_k} = (1 - q^2) q^{-3/2} b^* \pa_f(x)
\]
for all $x \in \C O(S_q^2)$.
\end{lemma}
\begin{proof}
The operation $x \mapsto [a^*,x]_{\de_k}$ satisfies a twisted Leibniz rule, meaning that
\[
[a^*, x y]_{\de_k} = [a^*,x]_{\de_k} y + \de_k(x) [a^*, y]_{\de_k} ,
\]
for all $x, y \in \C O(S_q^2)$. It moreover follows from \eqref{eq:taucommu} that the operation $x \mapsto b \pa_e(x)$ satisfies the same twisted Leibniz rule so that
\[
b \pa_1(xy) = b \pa_1(x) y + \de_k(x) b \pa_e(y) ,
\]
for all $x,y \in \C O(S_q^2)$. In order to prove the first identity of the lemma, it thus suffices to check that 
\[
a^* x - \de_k(x) a^* = (1 - q^2) q^{1/2} b \pa_e(x)
\]
for $x \in \{A,B,B^*\}$. This can be done in a straightforward fashion using that 
\begin{align}\label{eq:partial-on-generators}
\pa_e(A) &= - q^{-1/2} b^* a^* \notag\\ 
\pa_e(B) &= q^{-1/2} (b^*)^2\\
\pa_e(B^*) &= -q^{-3/2} (a^*)^2 \notag ,
\end{align}
which can be seen from \eqref{eq:derexp}. The second identity of the lemma follows by a similar argument, or, alternatively, from the first identity by applying the involution.
\end{proof}

We recall that
\[
u := u^1 = \ma{cc}{a^* & -qb \\ b^* & a} \in \mathbb{M}_2\big( \C O(SU_q(2)) \big)
\]
denotes the fundamental corepresentation unitary and that $\pa \colon \C O(S_q^2) \to \mathbb{M}_2(\C O(SU_q(2)))$ denotes the derivation
\[
\pa = \ma{cc}{0 & \pa_2 \\ \pa_1 & 0} = \ma{cc}{0 & q^{-1/2} \pa_f \\ q^{1/2} \pa_e & 0} .
\]

\begin{lemma}\label{l:deriIII}
We have the identities
\[
[u, x]_{\de_k} = (1 - q^2)  \ma{cc}{0 & b \\ q^{-1} b^* & 0} \pa(x) \quad \mbox{and} \quad
[u^*, \de_{k^{-1}}(x)]_{\de_k} = (1 - q^2)  \pa(x) \ma{cc}{0 & q^{-1} b \\ b^* & 0} 
\]
for all $x \in \C O(S_q^2)$.
\end{lemma}
\begin{proof}
Let $x \in \C O(S_q^2)$ be given. By definition of the fundamental corepresentation unitary $u$, the identities in \eqref{eq:taucommu}, and Lemma \ref{l:deriII} we see that
\[
\begin{split}
u x - \de_k(x) u & = \ma{cc}{ [a^*,x]_{\de_k} & -q [b,x]_{\de_k} \\ \,[b^*,x]_{\de_k} & [a,x]_{\de_k}} \\
& =  (1 - q^2) \ma{cc}{ b \pa_1(x) & 0 \\ 0 & q^{-1} b^* \pa_2(x)}
= (1 - q^2) \ma{cc}{0 & b \\ q^{-1} b^* & 0} \pa(x) .
\end{split}
\]
This proves the first identity of the lemma. The remaining identity then follows from the computation
\[
\begin{split}
(1-q^2)\pa(x) \ma{cc}{0 & q^{-1} b \\ b^* & 0} & = 
(q^2 - 1) \Big( \ma{cc}{0 & b \\ q^{-1} b^* & 0} \pa(x^*) \Big)^* 
= ( \de_k(x^*) u - u x^* )^* \\ 
& = u^* \de_k^{-1}(x) - x u^* = [u^* ,\de_k^{-1}(x)]_{\de_k} . \qedhere
\end{split}
\]
\end{proof}

We are now ready to show that the operation $x \mapsto u \pa(x) u^*$ is a twisted derivation on $\C O(S_q^2)$.

\begin{prop}\label{p:deri}
It holds that
\[
u \pa(xy) u^* = u \pa(x) u^* \de_k(y) + \de_{k^{-1}}(x) u \pa(y) u^*
\]
for all $x,y \in \C O(S_q^2)$.
\end{prop}
\begin{proof}
Let $x,y \in \C O(S_q^2)$ be given. We compute that
\[
\begin{split}
u \pa(xy) u^* & = u \pa(x) y u^* + u x \pa(y) u^* \\ 
& = u \pa(x) u^* \de_k(y) + \de_{k^{-1}}(x) u \pa(y) u^* 
- u \pa(x) [u^*,\de_k(y)]_{\de_{k^{-1}}} + [u,x]_{\de_{k^{-1}}} \pa(y) u^* ,
\end{split}
\]
so (upon conjugating with $u$) we have to show that
\[
u^* [u,x]_{\de_{k^{-1}}} \pa(y) = \pa(x) [u^*,\de_{k}(y)]_{\de_{k^{-1}}} u .
\]
Notice now that
\[
\ma{cc}{0 & q^{-1}b \\ b^* & 0} u = \ma{cc}{q^{-1} bb^* & ab \\ q^{-1} a^* b^* & - qb^* b} 
= u^* \ma{cc}{0 & b \\ q^{-1} b^* & 0} .
\]
Thus, applying Lemma \ref{l:deriIII} we obtain that
\[
\begin{split}
u^* [u,x]_{\de_{k^{-1}}} \pa(y) & = (x - u^* \de_{k^{-1}}(x) u) \pa(y) 
= - [u^*,\de_{k^{-1}}(x)]_{\de_k} u \pa(y) \\ 
& = (q^2 - 1) \pa(x) \ma{cc}{0 & q^{-1} b \\ b^* & 0}  u \pa(y)
= (q^2 - 1) \pa(x) u^* \ma{cc}{0 & b \\ q^{-1} b^* & 0} \pa(y)  \\
& = - \pa(x) u^* [u,y]_{\de_k} = \pa(x) [u^*,\de_k(y)]_{\de_{k^{-1}}} u .
\end{split}
\]
This proves the proposition.
\end{proof}

In order to finish our computation of the conjugated derivation $x \mapsto u \pa(x) u^*$ we introduce two twisted derivations $\de_1, \de_2\colon \C O(S_q^2) \to \C O(S_q^2)$ by setting
\[
\de_1 := q^{1/2} \de_e \quad \T{and} \quad \de_2 := q^{-1/2} \de_f.
\]
For $q \neq 1$, we furthermore define the twisted derivation $\de_3 := \frac{\de_k - \de_k^{-1}}{q - q^{-1}} \colon \C O(S_q^2) \to \C O(S_q^2)$ and for $q = 1$ we simply put $\de_3 :=  \frac12\de_h \colon \C O(S^2) \to \C O(S^2)$. Here, and below, the adjective ``twisted'' is to be understood in the sense of the following Leibniz type rule:
\[
\de_i(xy)=\de_i(x)\de_k(y) + \de_{k^{-1}}(x)\de_i(y),  
\]
for $x,y\in \C O(S_q^2)$ and $i\in \{1,2,3\}$; that this holds follows from \eqref{eq:delder}. For $q \in (0,1]$ we now assemble this data into a twisted derivation $\de \colon \C O(S_q^2) \to \mathbb{M}_2( \C O(S_q^2))$ by the formula
\[
\de(x) := \ma{cc}{ -\de_3(x) & \de_2(x) \\ \de_1(x) & \de_3(x)} .
\]
We record that $\de(x^*) = - \de(x)^*$, see  \eqref{eq:delstar} for this.

\begin{prop}\label{p:derV}
It holds that $u \pa(x) u^* = \de(x)$ for all $x \in \C O(S_q^2)$.
\end{prop}
\begin{proof}
Using Proposition \ref{p:deri} we see that the operations $x \mapsto u \pa(x) u^*$ and $x \mapsto \de(x)$ satisfy the same twisted Leibniz rule and they also behave in the same way with respect to the adjoint operation. It therefore suffices to verify the required identity on the generators $A,B \in \C O(S_q^2)$. This can be carried out by a straightforward computation using \eqref{eq:partial-on-generators} 
and \eqref{eq:delexp} together with the defining relations for $\C O(SU_q(2))$, indeed:
\[
\begin{split}
u \pa(A) u^* & = \ma{cc}{a^* & - qb \\ b^* & a} \ma{cc}{0 & ab \\ -b^* a^* & 0} \ma{cc}{a & b \\ -qb^* & a^*} \\
& = \ma{cc}{0 & a^* a ba^* + qbb^* a^* b \\
-q b^* a bb^* - ab^* a^* a & 0} \\
&= \ma{cc}{0 &  ba^* \\ -ab^* & 0} = \de(A), \quad \T{and} \\
u \pa(B) u^* & = \ma{cc}{a^* & - qb \\ b^* & a} \ma{cc}{0 & q^{-1} a^2 \\ (b^*)^2 & 0} \ma{cc}{a & b \\ -qb^* & a^*} \\
& = \ma{cc}{-a^* a^2 b^* - q b(b^*)^2 a & q^{-1} a^* a^2 a^* - q b^2(b^*)^2 \\ 0 & q^{-1} b^* a^2 a^* + a (b^*)^2 b} \\
& = \ma{cc}{-ab^* &  q^{-1}aa^* -qbb^*  \\ 0 &  ab^* } = \de(B) . 
\qedhere
\end{split}
\]
\end{proof}

We may now show that the Berezin transform is a Lip-norm contraction: 

\begin{prop}\label{p:derVI}
Let $q \in (0,1]$ and $N \in \nn$. The Berezin transform $\be_N \colon C(S_q^2) \to C(S_q^2)$ is a Lip-norm contraction in the sense that
\[
\| \partial \be_N(x) \| \leq \| \partial(x) \|
\]
for all $x \in \C O(S_q^2)$.
\end{prop}
\begin{proof}
Since $u \in \mathbb{M}_2\big( \C O(SU_q(2))\big)$ is unitary we obtain from Proposition \ref{p:derV} that
\[
\| \partial \be_N(x) \| = \| u \cd \partial \be_N(x) \cd u^* \|
= \| \de \be_N(x) \| .
\]
Then, since $\be_N \colon C(S_q^2) \to C(S_q^2)$ is a complete contraction we get from Lemma \ref{l:derI} that
\[
\| \de \be_N(x) \| = \| \be_N \de(x) \| \leq \| \de(x) \| = \| \pa(x) \| .
\]
This proves the result of the proposition.
\end{proof}

\section{Quantum Gromov-Hausdorff convergence}\label{sec:qGH-convergence}
The aim of the  present section is to prove our main convergence results, namely that the quantum fuzzy spheres $F_q^N$ converge to the Podle{\'s} sphere $S_q^2$ as the matrix size $N$ grows, and that the Podle{\'s} spheres $S_q^2$ converge to the classical 2-sphere (with its round metric) as $q$ tends to 1. However, before approaching the actual convergence results, quite a  bit of preparatory analysis is needed. As it turns out, the key to the convergence results is that the quantum Berezin transform $\beta_N$ provides a good approximation of the identity map on the Lip unit balls, and we prove this in the sections to follow. In the following section we provide the essential upper bound on $\|\beta_N(x)-x\|$ for $x$ in the Lip unit ball, and in the next section we prove that this upper bound indeed goes to zero as $N$ tends to infinity.

\subsection{Approximation of the identity}\label{ss:approx}
Throughout this section we let $N \in \nn_0$ and $q \in (0,1]$ be fixed. We let { $\epsilon \colon C(S_q^2) \to \cc$} denote the restriction of the counit to the Podle\'s sphere and we recall that { $h_N \colon C(S_q^2) \to \cc$} is the state given by $h_N(x) := h\big((a^*)^N x a^N \big) \cd \inn{N+1}$. 
We start out by proving an equivariance property for our derivations $\pa_1,\pa_2 \colon \C O(S_q^2) \to \C O(SU_q(2))$.

\begin{lemma}\label{l:approxI}
Let $x \in \C O(S_q^2)$. It holds that
\[
(1 \ot \pa_1)\De(x) = \De \pa_1(x) \quad \mbox{and} \quad (1 \ot \pa_2) \De(x) = \De \pa_2(x) .
\]
\end{lemma}
\begin{proof}
First note that since $\C O(S_q^2)$ is a left $\C O(SU_q(2))$-comodule the formulae in the lemma are well defined. 
As explained in Section \ref{sec:prelim}, the derivations $\pa_1, \pa_2 \colon \C O(S_q^2) \to \C O(SU_q(2))$ are given by
\[
\pa_1(y) = q^{1/2} (1 \ot \inn{e,\cd}) \De(y) \quad \T{and} \quad \pa_2(y) = q^{-1/2} (1 \ot \inn{f,\cd}) \De(y) 
\]
for all $y \in \C O(S_q^2)$. Using the coassociativity of $\De \colon \C O(SU_q(2)) \to \C O(SU_q(2)) \ot \C O(SU_q(2))$ we then obtain that
\[
(1 \ot \pa_1) \De(x) = q^{1/2}(1 \ot 1 \ot \inn{e,\cd}) (1 \ot \De) \De(x) = q^{1/2}(\De \ot \inn{e,\cd}) \De(x) = \De(\pa_1(x)) .
\]
A similar proof applies when $\pa_1$ is replaced by $\pa_2$.
\end{proof}

The next lemma is a consequence of the above Lemma \ref{l:approxI}, and standard properties of the minimal tensor product:

\begin{lemma}\label{l:approxII}
Let $\phi \colon C(SU_q(2)) \to \cc$ be a bounded linear functional. It holds that
\[
L_{D_q}\big( (\phi \ot 1) \De(x) \big) \leq \| \phi \| \cd L_{D_q}(x)
\]
for all $x \in \C O(S_q^2)$.
\end{lemma}

We let $d_q(h_N,\epsilon) \in [0,\infty)$ denote the Monge-Kantorovi\v{c} distance between the two states $h_N,\epsilon \colon C(S_q^2) \to \cc$. Recall that this is defined as 
\[
d_q(h_N,\epsilon):=\sup\big\{ |h_N(x)-\ep(x)| \mid x\in \C O(S_q^2), L_{D_q}(x)\leq 1\big\} .
\]

\begin{prop}\label{p:approxIII}
We have the inequality
\[
\| x - \be_N(x) \| \leq d_q(h_N,\epsilon) \cd L_{D_q}(x)
\]
for all $x \in \C O(S_q^2)$.
\end{prop}
\begin{proof}
Let $x \in \C O(S_q^2)$ be given. We record that $x = (1 \ot \epsilon) \De(x)$ and hence 
\[
x - \be_N(x) = \big( 1 \ot ( \epsilon - h_N) \big) \De(x) .
\]
Notice now that $x - \be_N(x) \in C(S_q^2) \su \B B\big( L^2(S_q^2) \big)$ and let $\eta , \ze \in L^2(SU_q(2))$ with $\| \eta \| , \| \ze \| \leq 1$ be given. We let $\phi_{\eta,\ze} \colon C(SU_q(2)) \to \cc$, $\phi_{\eta,\ze}(y) := \inn{\eta,\rho(y)\ze}$, denote the associated contractive linear functional, and compute that
\[
\phi_{\eta,\ze}( x - \be_N(x) )  = (\epsilon - h_N)( \phi_{\eta,\ze} \ot 1) \De(x) .
\]
Using the definition of the Monge-Kantorovi\v{c} distance together with Lemma \ref{l:approxII} we then obtain that
\[
\begin{split}
\big| \phi_{\eta,\ze}( x - \be_N(x) ) \big| 
& \leq d_q(h_N,\epsilon) \cd L_{D_q}\big( ( \phi_{\eta,\ze} \ot 1) \De(x) \big) \\
& \leq d_q(h_N,\epsilon) \cd \| \phi_{\eta,\ze} \| \cd L_{D_q}(x) 
\leq d_q(h_N,\epsilon) \cd L_{D_q}(x) .
\end{split}
\]
Taking the supremum over vectors $\eta$ and $\ze \in L^2(SU_q(2))$ with $\| \eta \|, \| \ze \| \leq 1$ we obtain that
\[
\| x - \be_N(x) \| \leq d_q(h_N,\epsilon) \cd L_{D_q}(x) . \qedhere
\]
\end{proof}

\subsection{Convergence to the counit}
Throughout this section we let $q \in (0,1]$ be fixed. We see from Proposition \ref{p:approxIII} that in order to establish that the Berezin transform provides a good approximation of the identity map on the Lip unit ball, we only need to verify that $\lim_{N\to \infty}d_q(h_N, \epsilon)=0$. Since we already know that $(C(S_q^2),L_{D_q})$ is a compact quantum metric space, the convergence in the Monge-Kantorovi\v{c} metric is equivalent to convergence in the weak$^*$-topology. 
We apply the notation $\C O(A,1) \su C(S_q^2)$ for the smallest unital $*$-subalgebra containing the element $A \in C(S_q^2)$.

\begin{prop}\label{prop:convergence-to-counit}
Let $q\in (0,1]$. The sequence of states $\{h_N\}_{N = 0}^\infty$ converges to $\epsilon : C(S_q^2) \to \cc$ in the weak$^*$-topology and hence $\displaystyle \lim_{N\to\infty} d_q(h_N,\epsilon)=0$.
\end{prop}
\begin{proof}
It suffices to show that
\begin{equation}\label{eq:limitstates}
\lim_{N\to \infty} h_N(x)=\epsilon(x),
\end{equation}
for all $x$ in the norm-dense unital $*$-subalgebra $\C O(S_q^2) \su C(S_q^2)$. Notice next that $h_N(x) = 0 = \epsilon(x)$ whenever $x = y B^i$ or $x = y (B^*)^i$ for some $i \in \nn$ and some $y \in \C O(A,1)$. Indeed, this is clear for the counit and for the state $h_N$ this follows since the Haar state is invariant under the modular automorphism $\nu : \C O(SU_q(2)) \to \C O(SU_q(2))$. When proving \eqref{eq:limitstates} we may thus restrict our attention to the case where $x \in \C O(A,1)$. Using next that
\[
\C O(A,1) = \T{span}_{\cc} \big\{ (a^*)^k a^k \mid k \in \nn_0 \big\}
\]
we only need to show that $h_N\big((a^*)^k a^k\big)$ converges to $\epsilon( (a^*)^k a^k\big) = 1$ for all $k \in \nn_0$. But this follows from the computation here below, where we utilize the fact that $(a^*)^{N+k}=u_{00}^{N+k}$ together with \eqref{eq:haarmatrixII}: 
\begin{align*}
 h_N\left((a^*)^k a^k\right)&=   \inn{N+1} h\big( (a^*)^{N+k} a^{N+k} \big)  = \tfrac{ \inn{N+1}}{ \inn{N + k + 1}}  = 1 - q^{2(N+1)} \tfrac{ \inn{k} }{ \inn{N + k + 1}} \underset{N\to \infty}{ \longrightarrow} 1 . \qedhere
\end{align*}
\end{proof}

\begin{cor}\label{cor:fuzzy-approx}
Let $q\in (0,1]$. For each $\ep>0$ there exists an $N_0\in \nn_0$ such that 
\[
\|x-\beta_N(x)\|\leq \ep \cd L_{D_q}(x)
\]
for all $N \geq N_0$ { and all $x \in \C O(S_q^2)$}.
\end{cor}
\begin{proof}
By Proposition \ref{p:approxIII} we have the inequality
\[
\| x - \be_N(x) \| \leq d_q(h_N,\epsilon) \cd L_{D_q}(x),
\]
and by Proposition \ref{prop:convergence-to-counit} we have that $\lim_{N\to \infty} d_q(h_N,\epsilon)=0$. 
\end{proof}

The above corollary { in combination with Proposition \ref{p:derVI}} now makes it an easy task to show that that the quantum fuzzy spheres $F^N_q$ converge to the Podle{\'s} sphere $C(S_q^2)$ as $N$ approaches infinity, and we carry out the details of this argument in Section \ref{subsec:quantum-fuzzy-to-podles}. However, in order to show that $C(S_q^2)$ converges to $C(S^2)$ as $q$ tends to $1$, we need to { estimate the distance $d_q(h_N,\epsilon)$ in a suitably uniform manner} with respect to the deformation parameter $q$.
As a first step in this direction we show that the states $h_N$ and $\epsilon \colon C(S_q^2) \to \cc$ may be restricted to the unital $C^*$-subalgebra $C^*(A,1) \su C(S_q^2)$ without changing their Monge-Kantorovi\v{c} distance. This will play out to our advantage since we already carried out a careful analysis of the compact quantum metric space $\big( C^*(A,1), L_{D_q} \big)$ in \cite{GKK:QI}. In fact, $C^*(A,1)$ is a commutative unital $C^*$-algebra and the Lip-norm $L_{D_q} \colon C^*(A,1) \to [0,\infty]$ comes from an explicit metric on the spectrum of the selfadjoint positive operator $A$. We shall give more details on these matters in the next section.\\
Letting $i : C^*(A,1) \to C(S_q^2)$ denote the inclusion we specify that
\begin{equation}\label{eq:metricres}
\begin{split}
d_q(h_N \ci i, \epsilon \ci i) & := \big\{ | h_N(x) - \epsilon(x) | \mid x \in \C O(A,1) \, , \, \, L_{D_q}(x) \leq 1 \big\} \quad \T{and} \\
d_q(h_N,\epsilon) & := \big\{ | h_N(x) - \epsilon(x) | \mid x \in \C O(S_q^2) \, , \, \, L_{D_q}(x) \leq 1 \big\} .
\end{split}
\end{equation}
In order to prove that these two quantities agree, we define the strongly continuous circle action $\si_R \colon S^1 \ti C(SU_q(2)) \to C(SU_q(2))$ by the formulae
\[
\si_R(z,a) := z a \quad \T{and} \quad \si_R(z,b) := z^{-1} b 
\]
This circle action gives rise to the operation $\Phi_0 \colon C(SU_q(2)) \to C(SU_q(2))$ given by
\[
\Phi_0(x) := \frac{1}{2 \pi} \int_0^{2\pi} \si_R(e^{it},x) \, dt .
\]
For $z \in S^1$ and $x \in \C O(S_q^2)$ we record the relation
\[
\si_R(z, \pa(x)) = \pa( \si_R(z,x)),
\]
which can be proved by a direct computation on the generators $A, B$ and $B^*$ and an application of the Leibniz rule. { Using that $\pa \colon \C O(S_q^2) \to \mathbb{M}_2\big(C(SU_q(2))\big)$ is closable (see the discussion after Theorem \ref{t:spectrip})} we then obtain the identity
\[
\pa( \Phi_0(x) ) = \Phi_0( \pa(x) )
\]
for all $x \in \C O(S_q^2)$. Remark that the restriction of $\Phi_0$ to $C(S_q^2)$ {yields a} conditional expectation onto $C^*(A,1)$ which maps the coordinate algebra $\C O(S_q^2)$ onto $\C O(A,1)$. These observations immediately yield the next result: 

\begin{lemma}\label{l:conmetII}
Let $q \in (0,1]$. We have the inequality
\[
L_{D_q}( \Phi_0(x)) \leq L_{D_q}(x)
\]
for all $x \in \C O(S_q^2)$. 
\end{lemma}

As alluded to above, we have the following:

\begin{lemma}\label{l:conmetIII}
Let $q \in (0,1]$. It holds that $d_q( h_N,\epsilon) = d_q( h_N \ci i,\epsilon \ci i)$ for all $N \in \nn_0$.
\end{lemma}
\begin{proof}
Let $N \in \nn_0$ be given. The inequality $d_q(h_N \ci i, \epsilon \ci i) \leq d_q(h_N,\epsilon)$ is clearly satisfied. To prove the remaining inequality, let $x \in \C O(S_q^2)$ with $L_{D_q}(x) \leq 1$ be given. Using Lemma \ref{l:conmetII} we then obtain that $L_{D_q}(\Phi_0(x))\leq 1$. Next, since $h_N\big( \si_R(z, x) \big) = h_N(x)$ and $\epsilon\big( \si_R(z,x) \big)$ it follows that $h_N(x)= h_N( \Phi_0(x))$ and $\epsilon(x) =\epsilon(\Phi_0(x) )$ and hence that 
\[
\big| h_N(x) - \epsilon(x) \big| = \big| (h_N \ci i) (\Phi_0(x)) - (\epsilon \ci i)(\Phi_0(x)) \big| \leq d_q( h_N \ci i, \epsilon \ci i) .
\]
This proves the present lemma.
\end{proof}

\subsection{Uniform approximation of the identity}\label{ss:conmet}
In this section we are no longer considering $q \in (0,1]$ to be fixed but rather as a variable deformation parameter. For this reason we decorate our generators $a,b \in \C O(SU_q(2))$ and $A,B \in \C O(S_q^2)$ with an extra subscript, e.g. writing $A_q$ instead of $A$. Likewise, we put
\[
\inn{n}_q := \inn{n} = \sum_{m = 0}^{n-1} q^{2n} .
\]
We are, however, fixing a $\de \in (0,1)$ and restrict our attention to the case where $q \in [\de,1]$. \\

We start out by shortly reviewing some of the results obtained in \cite{GKK:QI}. The unital $C^*$-subalgebra $C^*(A_q,1) \su C(S_q^2)$ is commutative and is therefore isomorphic to the continuous functions on the spectrum of $A_q$. For $q \neq 1$ this spectrum is given by $X_q := \{ q^{2m} \mid m \in \nn_0 \} \cup \{0\} \su [0,1]$ and for $q = 1$ the spectrum $X_q$ agrees with the whole closed unit interval $[0,1]$. We are from now on tacitly identifying the $C^*$-algebra $C^*(A_q,1)$ with the $C^*$-algebra $C(X_q)$.
It was proved in \cite{GKK:QI} that the restriction of the seminorm $L_{D_q} \colon C(S_q^2) \to [0,\infty]$ to $C^*(A_q,1)$ agrees with the Lipschitz constant seminorm arising from a metric $\rho_q : X_q \ti X_q \to [0,\infty)$. For $q \neq 1$ this metric is given by the explicit formula
\begin{align}\label{eq:metric-formula}
\rho_q(q^{2m}, q^{2l}) = \Big| \sum_{k = m}^\infty \frac{(1 - q^2) q^k}{\sqrt{1 - q^{2(k+1)}}}
- \sum_{k = l}^\infty \frac{(1 - q^2) q^k}{\sqrt{1 - q^{2(k+1)}}} \Big| 
\end{align}
for all $m,l \in \nn_0$. For $q = 1$, the metric is given by 
\[
\rho_1(s,t) = \big| \arcsin(2s-1) - \arcsin(2t-1) \big| \quad \T{for all } s,t \in [0,1] .
\]
We let $h_q \colon C(X_q) \to \cc$ denote the restriction of the Haar state to the $C^*$-subalgebra $C^*(A_q,1) \su C(SU_q(2))$. For $q = 1$ we remark that $h_1$ agrees with the usual Riemann integral on $C( [0,1])$. \\

Let $N \in \nn_0$. As mentioned earlier we are interested in estimating the Monge-Kantorovi\v{c} distance between the two states $h_N$ and $\epsilon \colon C(S_q^2) \to \cc$ in a suitably uniform manner with respect to the deformation parameter $q \in [\de,1]$. We now present an estimate on this quantity which involves the metric $\rho_q \colon X_q \ti X_q \to [0,\infty)$. 

\begin{lemma}\label{l:metricest}
Let $q \in [\de,1]$ and $N \in \nn_0$. We have the estimate
\[
d_q( \epsilon, h_N) \leq \frac{\inn{N+1}_q}{q^{2N}} \cd h_q\big( a_q^N (a_q^*)^N \cd \rho_q(-,0)\big) .
\]
\end{lemma}
\begin{proof}
By Lemma \ref{l:conmetIII} we have that $d_q(\epsilon,h_N) = d_q(\epsilon \ci i,h_N \ci i)$, where $i \colon C^*(A_q,1) \to C(S_q^2)$ denotes the inclusion. The composition $\epsilon \ci i \colon C^*(A_q,1) \to \cc$ agrees with the pure state on $C(X_q)$ given by  evaluation at $0 \in X_q$. Let now $\xi \in C(X_q)$ be given. Since the restriction $L_{D_q} \colon C(X_q) \to [0,\infty]$ agrees with the Lipschitz constant seminorm coming from the metric $\rho_q \colon X_q \ti X_q \to [0,\infty)$ we obtain that 
\[
\begin{split}
\big| \epsilon(\xi) - h_N(\xi) \big| 
& = \big| h_N\big( \xi(0) - \xi \big) \big| \leq h_N\big( | \xi(0) - \xi| \big)
\leq L_{D_q}(\xi) \cd h_N\big( \rho_q( -,0)\big) \\
& = L_{D_q}(\xi) \cd \frac{\inn{N+1}_q}{q^{2N}} \cd h_q\big( a_q^N (a_q^*)^N \cd \rho_q(-,0) \big)  ,
\end{split}
\]
where the last identity uses the definition of state $h_N \colon C(S_q^2) \to \cc$ together with the twisted tracial property of the Haar state, see \eqref{eq:statedef} and \eqref{eq:modular}. The result of the lemma now follows from the definition of the quantity $d_q(\epsilon \ci i,h_N \ci i) = d_q(\epsilon,h_N)$ as recalled in \eqref{eq:metricres}.
\end{proof}


We now carry out a more detailed analysis of the right hand side of the estimate appearing in Lemma \ref{l:metricest}. First of all we treat the dependency of the (restriction of the) Haar state $h_q \colon C(X_q) \to \cc$ on the deformation parameter $q \in [\de,1]$. Our next lemma can also be deduced from \cite[Page 195]{Bla:DCH}, but for the convenience of the reader we here provide a short self contained argument. Consider the polynomial algebra $\cc[x,y]$ as a unital $*$-subalgebra of $C\big( [\de,1] \ti [0,1] \big)$ and define the linear map $H \colon \cc[x,y] \to C( [\de,1])$ given by
\[
H( x^j y^k )(q) := \frac{q^j}{ \inn{k + 1}_q} .  
\]

\begin{lemma}\label{l:conmetIV}
The linear map $H \colon \cc[x,y] \to C( [\de,1])$ extends to a norm-contraction $H \colon C( [\de,1] \ti [0,1]) \to C([\de,1])$ such that $H( \xi)(q) = h_q( \xi(q, - ) |_{X_q} )$ for all $q \in [\de,1]$.
\end{lemma}
\begin{proof}
Let first $q \in [\de,1]$ be fixed. By definition of $H \colon \cc[x,y] \to C([\de,1])$ we have that
\[
H(x^j y^k)(q) = \frac{q^j}{ \inn{k+1}_q} = q^j \cd h_q( A_q^k) = h_q\big( (x^j \cd y^k)(q ,-)|_{X_q} \big)
\]
for all $j,k \in \nn_0$. Remark that the formula for $h_q(A_q^k)$ can be found in \cite[Chapter 4, Equation (51)]{KlSc:QGR}. Then, using that $h_q \colon C(X_q) \to \cc$ is a state for every $q \in [\de,1]$, we obtain that
\[
\begin{split}
\| H(\xi) \|_{\infty} & = \sup_{q \in [\de,1]} \big| H(\xi)(q) \big| 
= \sup_{q \in [\de,1]} \big| h_q( \xi(q ,\cd ) |_{X_q}) \big| 
 \leq \sup_{ (q,s) \in [\de,1] \ti X_q} \big| \xi(q,s) \big| \leq \| \xi \|_\infty 
\end{split}
\]
for all $\xi \in \cc[x,y]$. The result of the lemma now follows since $\cc[x,y] \su C( [\de,1] \ti [0,1])$ is dense in supremum norm. 
\end{proof}

Continuing our treatment of the right hand side of the estimate in Lemma \ref{l:metricest} we now analyze how the functions $\rho_q(-,0) \colon X_q \to [0,\infty)$ depend on the deformation parameter $q \in [\de,1]$. We record that
\begin{equation}\label{eq:distzero}
\begin{split}
\rho_q(q^{2m},0) 
& = \sum_{k = m}^\infty \frac{(1 - q^2) q^k}{\sqrt{1 - q^{2(k+1)}}} 
 = \sum_{k = 0}^\infty \frac{(1 - q^2) q^{k+m}}{\sqrt{1 - q^{2(k + m +1)}}} 
\end{split}
\end{equation}
whenever $q \neq 1$ and $m \in \nn_0$, see \eqref{eq:metric-formula}. In order to deal with these expressions we let $(q,s) \in [\de,1) \ti [0,1]$ and define the continuous decreasing function $\ze_{q,s} \colon (-1,\infty) \to [0,\infty)$ by the formula
\[
\ze_{q,s}(x) := \frac{ (1 - q^2) q^x \cd \sqrt{s} }{\sqrt{1 - s q^{2(x + 1)}}} 
\]
Each of these functions can be estimated from above as follows:
\[
\ze_{q,s}(x) \leq \frac{q^x}{\sqrt{1 - q^{2(x + 1)}} } \leq \frac{q^x}{ \sqrt{1 - q^2}} \quad \T{for all } x \geq 0 .
\]
We may thus introduce the function $f \colon [\de,1] \ti [0,1] \to [0,\infty)$ by putting
\begin{equation}\label{eq:conmetI}
f(q,s) := \fork{ccc}{
\sum_{k = 0}^\infty \ze_{q,s}(k) & \T{for} & q \neq 1 \\
2 \arcsin( \sqrt{s}) & \T{for} & q = 1
} .
\end{equation}
Comparing with the formula for the metric in \eqref{eq:distzero} we immediately see that
\begin{equation}\label{eq:metricfunc}
f(q,q^{2m}) = \sum_{k = 0}^\infty \ze_{q,q^{2m}}(k) = \sum_{k = 0}^\infty \frac{ (1 - q^2) q^{k+m} }{\sqrt{1 - q^{2(k + m + 1)}}} 
= \rho_q(q^{2m},0)
\end{equation}
whenever $q \neq 1$ and $m \in \nn_0$. Similarly, for $q = 1$ we have that
\begin{equation}\label{eq:metricfuncII}
f(1,s) = 2 \arcsin(\sqrt{s}) = \arcsin(2s-1) + \frac{\pi}{2} = \rho_1(s,0) \quad \T{for all } s \in [0,1].
\end{equation}

\begin{lemma}\label{l:conmetI}
The function $f \colon [\de,1] \ti [0,1] \to [0,\infty)$ is continuous.
\end{lemma}
\begin{proof}
We focus on proving continuity at a point of the form $(1,s_0)$ for a fixed $s_0 \in [0,1]$,  since the continuity of the restriction $f\vert_{[\de,1) \ti [0,1]}$ follows as the sequence of partial sums $\{ \sum_{k=0}^m \zeta_{-,-}(k)\}_{m = 0}^\infty$ converges in supremum norm to $f$ {on compact subsets of $[\de,1) \ti [0,1]$.} We remark that for each fixed $(q,s) \in [\de,1) \ti [0,1]$ it holds that the function $\ga_{q,s} \colon (-1,\infty) \to \rr$ given by
\[
\ga_{q,s} \colon x \mapsto \frac{1 - q^2}{\ln(q) \cd q} \cd \arcsin( \sqrt{s} \cd q^{x + 1} )
\]
is an antiderivative to $\ze_{q,s} \colon (-1,\infty) \to \rr$. Moreover, since $\ze_{q,s} \colon (-1,\infty) \to \rr$ is positive and decreasing we obtain the estimates
\[
\int_0^\infty \ze_{q,s}(x) dx \leq \sum_{k = 0}^\infty \ze_{q,s}(k) \leq \int_{-1}^\infty \ze_{q,s}(x) dx .
\]
In order to compute the above integrals we record the following formulae:
\[
\begin{split}
 & \lim_{x \to \infty} \ga_{q,s}(x) = 0  \quad \quad  \quad \quad \quad \quad \ga_{q,s}(0) = \frac{1 - q^2}{\ln(q) \cd q} \arcsin(\sqrt{s} \cd q) \\
 & \lim_{x \to -1} \ga_{q,s}(x) = \frac{1 - q^2}{\ln(q) \cd q} \arcsin(\sqrt{s}) .
 \end{split}
\]
We thereby obtain the estimates
\[
-\frac{1 - q^2}{\ln(q) \cd q} \arcsin(\sqrt{s} \cd q) \leq f(q,s) \leq -\frac{1 - q^2}{\ln(q) \cd q} \arcsin(\sqrt{s})
\]
for all $(q,s) \in [\de,1) \ti [0,1]$. The continuity of the function $f \colon [\de,1] \ti [0,1] \to [0,\infty)$ at the fixed point $(1,s_0) \in [\de,1] \ti [0,1]$ now follows by noting that
\[
\lim_{q \to 1} \frac{1 - q^2}{\ln(q) \cd q} = - 2. \qedhere
\]
\end{proof}

We are now ready for the final step regarding the continuity properties of the right hand side of the estimate in Lemma \ref{l:metricest}. For each $q \in [\de,1]$ and $N \in \nn_0$, we compute that
\[
a_q^N (a_q^*)^N 
= (1 - q^{-2(N-1)} A_q) \cd (1 - q^{-2(N-2)} A_q) \clc (1 - A_q) .
\]
We then define the continuous function $g_N \colon [\de,1] \ti [0,1] \to [0,\infty)$
\begin{equation}\label{eq:conmetII}
g_N(q,s) := \frac{\inn{N+1} }{q^{2N}} \cd (1 - q^{-2(N-1)} \cd s) \cd (1 - q^{-2(N-2)} s) \clc (1 - s)
\end{equation}
and note that $g_N(q,-)\vert_{X_q} = \frac{\inn{N + 1}}{q^{2N}} \cd a_q^N (a_q^*)^N$.\\
The next result summarises what we have obtained so far and is thus a consequence of Lemma \ref{l:metricest}, Lemma \ref{l:conmetIV} and Lemma \ref{l:conmetI} together with \eqref{eq:metricfunc}, \eqref{eq:metricfuncII} and \eqref{eq:conmetII}:


\begin{lemma}\label{l:conmetV}
For each $N \in \nn_0$ and each $q \in [\de,1]$ we have the estimate
\[
d_q( \epsilon,h_N) \leq H( f \cd g_N)(q) .
\]
\end{lemma}

We finish this section by proving two lemmas culminating in a uniform estimate on the distance between the Berezin transform and the identity map.

\begin{lemma}\label{l:conmetVI}
For each $q \in [\de,1]$, it holds that $\displaystyle \lim_{N \to \infty} H(f \cd g_N)(q) = 0$. 
\end{lemma}
\begin{proof}
Let $q \in [\de,1]$ be given. We first remark that it follows from the definition in \eqref{eq:conmetI} (see also \eqref{eq:metricfunc}) that $f(q,0) = 0$. Since the restriction of the counit to $C^*(A_q,1) \cong C(X_q)$ is given by evaluation at $0\in X_q$ this translates into the identity $\epsilon(f(q,\cdot)\vert_{X_q})=0$. Moreover, we notice that
\[
H(f \cd g_N)(q)= h_q\big( g_N(q,-)\vert_{X_q} \cd f(q,-)\vert_{X_q}\big) = h_N\big( f(q,-)\vert_{X_q} \big) .
\]
The result of the lemma now follows since the sequence of states $\{h_N\}_{N = 0}^\infty$ converges to the restriction of the counit $\epsilon \colon C(S_q^2) \to \cc$ in the weak$^*$-topology by Proposition \ref{prop:convergence-to-counit}.
\end{proof}

\begin{lemma}\label{l:conmetVII}
For each $\ep > 0$, each $q_0 \in [\de,1]$ and each $N_0 \in \nn_0$ there exists an $N \geq N_0$ and an open interval $I \su \rr$ with $q_0 \in I$ such that
\[
\| x - \be_N(x) \| \leq \ep \cd L_{D_q}(x)
\]
for all $q \in I \cap [\de,1]$ and all $x \in \C O(S_q^2)$.
\end{lemma}
\begin{proof}
Let $\ep > 0$, $q_0 \in [\de,1]$ and $N_0 \in \nn_0$ be given. By Lemma \ref{l:conmetVI} we may choose an $N \geq N_0$ such that $H(f \cd g_N)(q_0) < \ep/2$. Then, using Lemma \ref{l:conmetI} and Lemma \ref{l:conmetIV}, we see that $H(f \cd g_N) \in C([\de,1])$. We may therefore choose our open interval $I \su \rr$ such that $\big| H(f \cd g_N)(q) - H(f \cd g_N)(q_0)\big| < \ep/2$ for all $q \in I \cap [\de,1]$. Combining this result with Lemma \ref{l:conmetV} and Proposition \ref{p:approxIII} we then obtain that
\[
\| x - \be_N(x) \| \leq d_q(h_N,\epsilon) \cd L_{D_q}(x) 
\leq \big| H(f \cd g_N)(q) \big| \cd L_{D_q}(x)  \leq \ep \cd L_{D_q}(x)
\]
for all $q \in I \cap [\de,1]$ and all $x \in \C O(S_q^2)$. This ends the proof of the lemma.
\end{proof}


\subsection{Quantum fuzzy spheres converge to the fuzzy sphere}
In this section we prove that the quantum fuzzy spheres converge to the classical fuzzy sphere as the deformation parameter $q$ tends to $1$. As remarked in Section \ref{ss:quafuz}, rather than thinking of the fuzzy sphere as a matrix algebra we will consider its image $F^N_1:=\sigma_N(\mathbb{M}_{N+1}(\cc))$ under the covariant Berezin symbol $\sigma_N \colon \mathbb{M}_{N+1}(\cc)\to C(S^2)$, see Lemma \ref{lem:fuzzy-equal-image} and the discussion after this lemma. We recall that each of the quantum fuzzy spheres $F^N_q \su C(S_q^2)$ is equipped with the Lip-norm arising from restricting the Lip-norm $L_{D_q} \colon C(S_q^2) \to [0,\infty]$, and in this way each $F^N_q$ becomes a compact quantum metric space. 


\begin{prop}\label{prop:quantum-fuzzy-to-classical-fuzzy}
Let $N \in \nn_0$. The quantum fuzzy spheres $\left(F^N_q\right)_{q\in (0,1]}$ vary continuously in the parameter $q$ with respect to the quantum Gromov Hausdorff distance.
\end{prop} 
\begin{proof}
Let $\de \in (0,1)$ be given. By \cite[Proposition 7.1]{Bla:DCH}, there exists a continuous field of $C^*$-algebras over $[\de,1]$ with total space $C(SU_\bullet(2))$ such that $C(SU_q(2))$ agrees with the fibre at $q \in [\de,1]$. There exist continuous sections of this field $A_\bullet, B_\bullet \in C(SU_\bullet(2))$ mapping to the generators $A_q$ and $B_q \in C(S_q^2) \su C(SU_q(2))$ under the quotient map $\T{ev}_q \colon C(SU_\bullet(2)) \to C(SU_q(2))$ for each $q \in [\de,1]$. Let us define the subspace
\[
V := \T{span}_\cc\big\{ A^i_\bullet B_\bullet^j, A_\bullet^i (B_\bullet^*)^j \mid i,j \in \nn_0 \, , \, \, i + j \leq N \big\}
\su C(SU_\bullet(2))
\]
and let $V_{\sa} \su V$ denote the real part $V_{\sa} := \big\{ x + x^* \mid x \in V \big\}$ so that $V_{\sa}$ becomes a real vector space of dimension $(N + 1)^2$ containing the unit $1$ from $C(SU_\bullet(2))$. \\
We remark that it follows from Definition \ref{d:quantumfuzz} that the image $\T{ev}_q(V_{\sa})$ agrees with the order unit space $(F_q^N)_{\sa}$ for each $q \in [\de,1]$. Therefore, upon defining $\| y\|_q := \| \T{ev}_q(y) \|$ for each $q \in [\de,1]$ we obtain a continuous field $\{ \| \cd \|_q\}_{q \in [\de,1]}$ of order unit norms on $(V_{\sa},1)$. For each $q \in [\de,1]$ we now record the formulae $\pa_1(A_q) = -b_q^* a_q^*$, $\pa_1(B_q) = (b^*_q)^2$ and $\pa_1(B_q^*) = - q^{-1} (a_q^*)^2$. Since we have continuous sections $a_\bullet, b_\bullet \in C(SU_\bullet(2))$ (mapping to the generators of $C(SU_q(2))$ under the quotient map $\T{ev}_q$) we then obtain a continuous family of Lip-norms $\{ L_q \}_{q \in [\de,1]}$ on $V_{\sa}$ defined by $L_q(y) := L_{D_q}( \T{ev}_q(y))$ for each $q \in [\de,1]$. The assumptions in \cite[Theorem 11.2]{Rie:GHD} are thereby fulfilled, and since $\de \in (0,1)$ was arbitrary this implies the claimed continuity result.
\end{proof}

\subsection{Quantum fuzzy spheres converge to the Podle\'s sphere}\label{subsec:quantum-fuzzy-to-podles}
In this section we fix a $q \in (0,1]$ and show that the quantum fuzzy spheres $F^N_q$ converge to $C(S_q^2)$ as $N$ tends to infinity. This follows directly from our analysis of the quantum Berezin transform and the following lemma, which is certainly part of the folklore knowledge, but seems not to be directly available in the literature. We remark that when the quantum Gromov-Hausdorff distance is replaced by Latr{\'e}moli{\`e}re's quantum propinquity, the statement can be found in \cite[Theorem 6.3]{Lat:QGH}, but for the benefit of the reader, we include a proof for the corresponding statement in our setting. {We emphasise that the seminorm $K \colon Y \to [0,\infty]$ appearing in the statement need not agree with $L|_Y$ -- the seminorms $K$ and $L$ only agree on the dense domain of $K$. 

\begin{lemma}\label{lem:order-unit-version-of-frederics-result}
Let $(X,L)$ be a compact quantum metric space and let $Y\su X$ be a sub-operator system equipped with a seminorm $K \colon Y \to [0,\infty]$ with dense $*$-invariant domain $\T{Dom}(K)\subseteq \T{Dom}(L)$ such that $K(y) = L(y)$ for all $y \in \T{Dom}(K)$. Let $\ep>0$. If for every $x\in X$ there exists $y\in Y$ such that $K(y)\leq L(x)$ and $\|x-y\| \leq \ep L(x)$, then $\dist_{\T{Q}}\big((X,L); (Y,K) \big)\leq \ep$.
\end{lemma}
Note that an application of Theorem \ref{t:totallybdd} gives that $(Y,K)$ automatically becomes a compact quantum metric spaces, so that the statement in  the lemma indeed makes sense.

\begin{proof}
As explained in Section \ref{subsec:cqms}, the compact quantum metric spaces $(X,L)$ and $(Y,K)$ give rise to order unit compact quantum metric space denoted by $A:=\{x\in X_{\sa} \mid L(x)<\infty\}$ and $B:=\{y\in Y_\sa \mid K(y)<\infty\}$. And in fact, since
\[
\dist_{\T{Q}}\big((X,L); (Y,K)\big)= \dist_{\T{Q}}\big((A,L); (B ,K)\big)  
\]
one may as well pass to the order unit setting when considering matters related to the quantum Gromov-Hausdorff distance. \\
We define a seminorm $L_\ep$ on $A \oplus B$ by setting
\[
L_\ep(a,b):=\max\left\{L(a), \, K(b), \, \tfrac{1}{\ep}\|a-b\| \right\}.
\]
By construction, $L_\ep(a,b)=0$ if and only if $(a,b)=t(1,1)$ for some $t\in \rr$. By assumption, for each $a\in A$ there exists $b \in B$ such that $K(b)\leq L(a)$ and $\|a-b\|\leq \ep L(a)$ and hence
\[
L_\ep(a,b) \leq L(a).
\]
Conversely, if $b \in B$ we trivially have that $L_\ep(b,b) \leq L(b)$ (using now that $K(b) = L(b)$). This proves that the assumptions in \cite[Theorem 5.2]{Rie:GHD} are fulfilled, and from this we obtain that $L_\ep \colon A\oplus B \to [0,\infty)$ is an admissible Lip-norm. \\
Next, denote by $\pi_1\colon A\oplus B \to A$ and $\pi_2\colon A\oplus B \to B$ the natural projections, and note that by  \cite[Proposition 3.1]{Rie:GHD}, the dual maps $\pi_1^* \colon  \C S(A) \to \C S(A \oplus B)$ and  $\pi_2^*\colon  \C S(B) \to \C S(A \oplus B)$ are isometries for the associated Monge-Kantorovi\v{c} metrics. Given $\psi \in \C S(B)$, we extend $\psi$ to a state $\tilde{\psi}$ on $A$ using the Hahn-Banach theorem, see \cite[Chapter II (1.10)]{alfsen-book}. For $(a,b) \in A\oplus B$ with $L_\ep (a,b) \leq 1$, we have $\| b-a\|\leq \ep$ and thus
\[
\big|\pi_2^*(\psi)(a,b)-\pi_1^*(\tilde{\psi})(a,b)\big| = \big|\psi(b)-\tilde{\psi}(a) \big|
 = \big|\tilde{\psi}(b-a)\big| \leq \|b-a\| \leq \ep.
\] 
Hence $\rho_{L_{\ep}}\big(\pi_2^*(\psi), \pi_1^*(\tilde{\psi})\big)\leq \ep$. Conversely, given $\varphi \in \C S(A)$ we have $\varphi\vert_B \in \C S(B)$ and an analogous computation proves the inequality $\rho_{L_{\ep}}\big(\pi_2^*(\varphi\vert_B), \pi_1^*(\varphi)\big)\leq \ep$. Hence, we have shown that
\[
\dist_H^{\rho_{L_\ep}}\big(\pi_1^*(\C S(A)), \pi_2^*(\C S(B))\big)\leq \ep,
\]
and therefore $\dist_{\T{Q}}\big((A,L);(B,K) \big)\leq \ep$ as desired. 
\end{proof} }

Our result regarding convergence of the quantum fuzzy spheres towards the Podle\'s sphere can now be stated and proved. We emphasise that the domain of our Lip-norm $L_{D_q} \colon C(S_q^2) \to [0,\infty]$ is given by the coordinate algebra $\C O(S_q^2)$. 

\begin{thm}\label{thm:fuzzy-to-podles}
For each $q\in (0,1]$ it holds that $\displaystyle\lim_{N\to \infty} \dist_{\T{Q}}\big(F^N_q; C(S_q^2)\big)=0$.
\end{thm}
\begin{proof}
This follows from Proposition \ref{p:derVI}, Corollary \ref{cor:fuzzy-approx} and Lemma \ref{lem:order-unit-version-of-frederics-result}.
\end{proof}

Note that when $q=1$, Theorem \ref{thm:fuzzy-to-podles} gives a variation of Rieffel's original result \cite[Theorem 3.2]{Rie:MSG}, but since our Lip-norm on $F^N_1$ is { a priori} different from the one considered in \cite{Rie:MSG}, we do not recover the classical result verbatim.

\subsection{The Podle\'s spheres converge to the sphere}
In this section we prove the main result of the paper, which at this point follows rather easily from the analysis carried out in the previous sections. 
\begin{thm}\label{thm:podles-converging-to-classical}
For any $q_0\in(0,1]$ one has $\displaystyle\lim_{q\to q_0} \dist_{\T{Q}}\big( C(S_q^2); C(S_{q_0}^2)\big)=0$.
\end{thm}
\begin{proof}
Let $\ep>0$ be given. By Proposition \ref{p:derVI}, Lemma \ref{l:conmetVII} and Lemma \ref{lem:order-unit-version-of-frederics-result} there exists an open interval $I$ containing $q_0$ and an $N\in \nn_0$ such that $\dist_{\T{Q}}(F_q^N; C(S_q^2))<\ep/3$ for all $q\in I\cap (0,1]$. Upon shrinking $I$ if needed, Proposition \ref{prop:quantum-fuzzy-to-classical-fuzzy} shows that we may assume $\dist_{\T{Q}}(F^N_q; F^N_{q_0})<\ep/3$ for all $q\in I \cap (0,1]$. Given $q \in I$, the inequality $\dist_{\T{Q}}\big(C(S_q^2);C(S_{q_0}^2) \big)<\ep$ now follows from the triangle inequality for the quantum Gromov-Hausdorff distance, see \cite[Theorem 4.3]{Rie:GHD}.
\end{proof}

\begin{remark}
Theorem \ref{thm:podles-converging-to-classical} is, in reality, a result about the coordinate algebras $\C O(S_q^2)$, $q \in (0,1]$, in so far that these are exactly the domains of the Lip-norms $L_{D_q} \colon C(S_q^2) \to [0,\infty]$. These coordinate algebras are the natural domains from a Hopf-algebraic point of view. Another natural (and much larger) domain would be the Lipschitz algebra $C^{\Lip}(S_q^2)$, and,  in fact, the above convergence result holds true for the Lip-norms $L_{D_q}^{\T{max}} \colon C(S_q^2) \to [0,\infty]$ with domain  $C^{\Lip}(S_q^2)$ discussed in Remark \ref{r:lipspec}. This result relies on Theorem \ref{thm:podles-converging-to-classical} but requires a substantial amount of extra analysis. For this reason we defer the details to the separate paper \cite{AKK:DistZero}. In fact, we show in \cite[Theorem A]{AKK:DistZero} that the quantum Gromov-Hausdorff distance between $(C(S_q^2), L_{D_q})$ and $(C(S_q^2), L_{D_q}^{\T{max}})$ is equal to zero for each $q \in (0,1]$.
\end{remark}

\bibliographystyle{plain}

\end{document}